\documentclass[11pt]{article}

\usepackage{amsmath,amsthm}

\usepackage{amssymb,latexsym}

\usepackage{enumerate}

\newtheorem{tw}{Theorem}[subsection]
\newtheorem{lm}[tw]{Lemma}
\newtheorem{wn}[tw]{Corollary}
\newtheorem{stw}[tw]{Proposition}
\newenvironment{dow}{\it Proof.\rm}{\hfill $\Box$}

\theoremstyle{definition}
\newtheorem*{df}{Definition}
\newtheorem{uw}[tw]{Remark}
\newtheorem{prz}[tw]{Example}

\newcommand{\BN}{{\mathbb N}}
\newcommand{\BR}{{\mathbb R}}
\newcommand{\BX}{{\mathbb X}}

\newcommand{\FF}{{\mathcal{F}}}
\newcommand{\GG}{{\mathcal{G}}}

\newcommand{\BB}{{\mathcal{B}}}
\newcommand{\LL}{{\mathcal{L}}}

\newcommand{\DD}{{\mathbf{D}}}

\newcommand{\PP}{{\mathcal{P}}}
\newcommand{\EE}{{\mathcal{E}}}

\newcommand{\1}{{\mathbf{1}_{[0,\tau_{k}]}}}

\newcommand{\nsubsection}{\setcounter{equation}{0}\subsection}

\renewcommand{\thefootnote}{{}}

\begin{document}
\title {Right Markov processes and systems of semilinear equations
with measure data}
\author {Tomasz Klimsiak }
\date{}
\maketitle
\begin{abstract}
In the paper we prove the existence of probabilistic solutions to
systems of the form $-Au=F(x,u)+\mu$, where $F$ satisfies a
generalized sign condition and $\mu$ is a smooth measure. As for
$A$ we assume that it is a generator of a Markov semigroup
determined by a right Markov process whose resolvent is order
compact on $L^1$. This class includes local and nonlocal operators
corresponding to Dirichlet forms  as well as some operators which
are not in the variational form. To study the problem we introduce
new concept of compactness property relating the underlying Markov
process to almost everywhere convergence. We prove some useful
properties of the compactness property and  provide its
characterization  in terms of Meyer's property (L) of Markov
processes and in terms of order compactness of the associated
resolvent.
\end{abstract}

\renewcommand{\thefootnote}{}

\footnote{T. Klimsiak} \footnote{Faculty of Mathematics and
Computer Science,  Nicolaus Copernicus University, Chopina 12/18,
87--100 Toru\'n, Poland}

\footnote{e-mail: tomas@mat.uni.torun.pl,\quad tel.  +48 56 611 3410 \quad
fax: +48 56 611 2987}

\renewcommand{\thefootnote}{\arabic{footnote}}
\setcounter{footnote}{0}

\nsubsection{Introduction}

Let $E$ be  a Radon metrizable topological space, $F:E\times
\BR^N\rightarrow\BR^N$, $N\ge1$, be a measurable function and let
$\mu=(\mu_1,\dots,\mu_N)$ be a smooth measure on $E^N$. In the
present paper we investigate the problem of existence of solutions
of the system
\begin{equation}
\label{eqi.1} -Au=F(x,u)+\mu.
\end{equation}
Here $A$ is the linear operator associated with a Markov semigroup
$\{T_t,t\ge 0\}$ on $L^1(E;m)$. Our only assumption on $\{T_t\}$
is that it is representable by some right Markov process
$\mathbb{X}=(\{X_t,t\ge 0\},\{P_x,x\in E\})$ on $E$, i.e. for
every $t\ge 0$ and $f\in L^1(E;m)$,
\begin{equation}
\label{eq1.2} (T_tf)(x)=E_xf(X_t)\equiv p_tf(x)\quad \mbox{for
$m$-a.e. } x\in E,
\end{equation}
where $E_x$ denotes the expectation with respect to the measure
$P_x$. The class of operators associated with such semigroups is
fairly wide. It includes important local and nonlocal operators
corresponding to quasi-regular Dirichlet forms  (see
\cite{Fukushima,MR,Oshima}) as well as interesting operators which
are not in the variational form,  like some classes of
Ornstein-Uhlenbeck processes (see Example 5.7).

As for $F=(f_1,\dots,f_N)$ we assume that it is continuous with
respect to $u$ and satisfies the following  sign condition:
\begin{equation}
\label{eq1.3} \langle F(x,y),y\rangle \le G(x)|y|,\quad x\in
E,\,y\in\mathbb{R}^N
\end{equation}
for some appropriately integrable positive function $G$ (see
hypotheses (H1)--(H4) in Section \ref{sec3}).

The first problem we encounter when dealing with systems of the
form (\ref{eqi.1}) is to give suitable definition of a solution.
The problem occurs even in the case of one linear equation with
local operator of the form $A=\sum_{i,j=1}^{d}
\frac{\partial}{\partial x_j}(a_{ij}\frac{\partial}{\partial
x_i})$, whose study goes back to the papers of Serrin
\cite{Serrin} and Stampacchia \cite{Stampacchia}. Serrin
\cite{Serrin} constructed an example of (discontinuous)
coefficients $a_{ij}$ and nontrivial function $u$ having the
property that $u\in W^{1,q}_0(D)$ for every $q<d/(d-1)$ and $u$ is
the distributional solution of (\ref{eqi.1}) with data $\mu=0$,
$F=0$. Since it was known that in general one can not expect  that
a solution to (\ref{eqi.1}) belongs to the space $W^{1,q}_0(D)$
with $q\ge d/(d-1)$, the problem of the  definition of a solution
to (\ref{eqi.1}) ensuring uniqueness arose. Stampacchia
\cite{Stampacchia} solved this problem by introducing the
so-called definition by duality. Since his work the theory of
scalar equations with measure data and local operators (linear and
nonlinear of of Leray-Lions type) have attracted considerable
attention (see \cite{Betal.,BGO,BDGO,DMO,DPP} for results for
equations with smooth measures $\mu$; a nice  account of the
theory for equations with general measures has been given in
\cite{BBr}).

The case of nonlocal operators is much more involved. To our
knowledge there were only few attempts to investigate scalar
linear equation (\ref{eqi.1}) with operator $A=\Delta^{\alpha}$
with $\alpha\in (0,1]$ by analytical methods (see \cite{AAB,KPU}).
To encompass broader class of operators and semilinear equations
in \cite{KR:JFA} (see also \cite{KR:MN}) a probabilistic
definition of a solution of scalar problem (\ref{eqi.1}) is
proposed. The basic idea in \cite{KR:JFA} is to define a solution
via  a nonlinear Feynman-Kac formula. Namely, a solution of
(\ref{eqi.1}) is a measurable function $u:E\rightarrow \mathbb{R}$
such that
\begin{equation}
\label{eqi.2} u(x)=E_x\int_0^{\infty}F(X_t,u(X_t))\,dt
+E_x\int_0^\infty dA^{\mu}_t
\end{equation}
for $m$-a.e. $x\in E$, where $A^{\mu}$ is a continuous additive
functional of the process $\BX$ corresponding to the measure $\mu$
in the Revuz sense (see \cite{FG,Fukushima,MR,Revuz}). In
\cite{KR:JFA} it is proved that in case $N=1$ if $F$ is
nonincreasing with respect to $u$ then under mild integrability
assumptions on the data there exists a unique solution to
(\ref{eqi.1}). In fact,  if $A$ is a uniformly divergence form
operator then the probabilistic solution of (\ref{eqi.1})
coincides with Stampacchia's solution by duality.

When studying systems (\ref{eqi.1}) with $F$ satisfying merely
sign condition (\ref{eq1.3}) we encounter new difficulties, which
roughly speaking pertain to weaker regularity of solution of
(\ref{eqi.1})  then in the scalar case and to ``compactness
properties". In \cite{Kl:AMPA} we have studied systems of the form
(\ref{eqi.1}) on bounded domain $D\subset \mathbb{R}^d$ with
$A=\Delta$ subject to homogeneous Dirichlet boundary condition. In
\cite{Kl:AMPA} it is observed that in general, if $F$ only
satisfies the sign condition, one cannot expect that
$F(\cdot,u)\in L^1(D;m)$. Moreover,  it may happen that the first
integral on the right-hand side of (\ref{eqi.2}) is infinite. This
together with the comments given before show that for systems,
even in the case of a uniformly elliptic divergence form operator,
neither the distributional definition nor the probabilistic via
the Feynman-Kac formula (\ref{eqi.2}) are applicable. For these
reasons in \cite{Kl:AMPA} more general than in \cite{KR:JFA,KR:MN}
probabilistic definition of a solution of (\ref{eqi.1}) is
adopted. It uses the representation of $u$ in terms of some
backward stochastic differential equation (BSDE) associated with
$A,F,\mu$ (in case $F(\cdot,u)$ is integrable the representation
reduces to (\ref{eqi.2})). This approach via BSDEs only  requires
quasi-integrability of $F(\cdot,u)$. It turns out that this mild
demand is always satisfied for solutions of (\ref{eqi.1}).
Therefore in the present paper we use some suitable generalization
of the definition from \cite{Kl:AMPA} (see Section \ref{sec3}).

As for ``compactness properties", let us note that in
\cite{KR:JFA} it is shown that if $N=1$ and $F$ is nonincreasing
then for $A$ associated with a Dirichlet form the function
$F(\cdot,u)$ is integrable but in general, $u$ is not integrable
(even locally). Since in case $N\ge2$ also the function
$F(\cdot,u)$ need not be integrable, it is fairly unclear what
type of function space possessing Banach structure to use to get
the existence result for (\ref{eqi.1}). In \cite{Kl:AMPA} we have
used the specific structure of the operator $A=\Delta$ to prove
that a solution of (\ref{eqi.1}) equals locally (i.e. on some
finely open sets) to some function from $H^1_0(D)$, which allowed
us to apply the Rellich-Kondrachov theorem on finely open sets
(see also \cite{Fuglede1,Fuglede2} for the theory of Laplacians on
finely open domains). In general, this approach fails. To overcome
the difficulty, in the present paper we introduce a notion of
compactness property relating the process $\BX$ to given solid
$\PP$ and positive subadditive set function $m$ on $E$ (not
necessarily measure). The compactness property is intended to
study $m$-a.e. convergence of sequences of functions defined on
$E$, pointwise convergence (when $m$ is a counting measure) and
quasi-everywhere convergence (when $m$ is the capacity determined by
$A$). It appears that such analysis of pointwise behaviour of
sequences of functions, in particular sequences of the form
$\{p_tf_n\}$, $\{R_{\alpha}f_n\}$, where $p_tf$ is defined by
(\ref{eq1.2}) and $R_{\alpha}f$ is the probabilistic resolvent
defined by
\begin{equation}
\label{eq1.4} R_\alpha f(x)=E_x\int_0^{\infty}e^{-\alpha
t}f(X_t)\,dt,\quad x\in E,
\end{equation}
is sufficient for the proof of existence of probabilistic
solutions to (\ref{eqi.1}).
%%It well suits our
%purpose of showing the existence of probabilistic solutions to
%(\ref{eqi.1}) and appears to be invariant with respect to
%important transformations of the process $\BX$.

Roughly speaking, given a solid $\PP\subset \mathcal{B}^+(E)$ and
a positive subadditive set function $m$ on $E$ we say the triple
$(\BX,\PP,m)$ has the compactness property  if for some $\alpha>0$
the probabilistic resolvent (\ref{eq1.4}) carries the family $\PP$
in a relative compact set in the topology of $m$-a.e. convergence
(see Section 2.1). If $m$ is the counting measure  then we will
omit $m$ in the notation and simply say that $(\BX,\PP)$ has the
compactness property.

In applications the family $\PP=\{u\in \mathcal{B}^+(E); u\le
1\}\equiv\BB_1$ plays pivotal role. Mokobodzki (see \cite[Section
XII, Theorem 89]{DellacherieMeyer}) has observed  that if $\BX$
satisfies hypothesis (L) of Meyer then $(\BX,\BB_1)$ has the
compactness property. Actually, he has observed that $R_\alpha
:\mathcal{B}_b(E)\rightarrow \mathcal{B}_b(E)$ is compact if we
equip $\mathcal{B}_b(E)$ with the topology of uniform convergence.
From \cite[Proposition 5.2, page 32]{Revuz} it follows that in
fact compactness of $R_\alpha:\mathcal{B}_b(E)\rightarrow
\mathcal{B}_b(E)$ with $\mathcal{B}_b(E)$ equipped with the
topology of uniform convergence is equivalent to hypothesis (L) of
Meyer. In Section \ref{sec2} using results of
\cite{WalshWinkler,WeisWerner} we prove that
\[
(\BX,\BB_1)\mbox{ has the compactness property iff $\BX$ satisfies
Meyer's hypothesis (L)}.
\]
In Section \ref{sec4} we show that if $m$ is an excessive measure
then
\begin{align}
\label{eq1.5}
&(\BX,\BB_1,m)\mbox{ has the compactness property iff}\nonumber\\
&\qquad\qquad\quad R_\alpha: L^1(E;m)\rightarrow L^1(E;m)\mbox{ is
order compact for some }\alpha>0.
\end{align}
Here by order compactness we mean that for every positive $v\in
L^1(E;m)$, $R_\alpha$ carries order intervals $[0,v]=\{u\in
L^{1}(E;m): 0\le u\le v\}$ in relatively compact subsets of
$L^1(E;m)$. We also investigate some stability properties of the
compactness property with respect to  transformation of the
underlying process. The most important result in this direction is
Proposition \ref{stw.2}. It says that for every
$B\in\mathcal{B}(E)$, if $(\BX,\PP,m)$ has the compactness
property then $(\BX^B,\PP(B),m)$ has the compactness property,
where $\BX^B$ denotes the part of $\BX$ on $B$ and
$\PP(B)=\{u\in\PP; u(x)= 0,x\in E\setminus B\}$. We have already
mentioned that it is reasonable to expect that $F(\cdot,u)$ and
$u$ are quasi-integrable which roughly speaking means that they
are integrable on subsets of $E$ whose complements have small
capacity naturally generated by the operator $A$. The significance
of Proposition \ref{stw.2} is that it allows to reduce the proof
of existence of solutions of (\ref{eqi.1}) to the analysis of the
system  (\ref{eqi.1}) on such  sets. Let us also note that in some
sense Proposition \ref{stw.2} resembles results on compactness of
positive operators subordinated to compact operator (see \cite{AB}
and Corollary \ref{wn2.222}).

The second problem that we address in Section \ref{sec2} is to
find conditions on a sequence $\{u_n\}$ of functions on $E$, which
together with the compactness property imply that $\{u_n\}$ is
relatively compact in the topology of $m$-a.e. convergence.  Our
main result is Theorem \ref{stw2.1}, which says that if
($\BX,\PP,m)$ has the compactness property and $\{u_n\}\subset\PP$
satisfies the condition
\begin{equation}
\label{eq1.7} \lim_{t\rightarrow 0^+}\sup_{n\ge 1}
|p_tu_n(x)-u_n(x)|=0\quad \mbox{for $m$-a.e. }x\in E
\end{equation}
then $\{u_n\}$ has a subsequence convergent $m$-a.e. Condition
(\ref{eq1.7}) is satisfied for instance if for $m$-a.e. $x\in E$
the sequence of processes $\{u_n(X)\}$ is tight in the Skorokhod
topology $J_1$ under the measure $P_x$. It is worth noting here
that in the paper the notion of compactness of a triple
$(\mathbb{X},\PP,m)$ is defined for general normal processes (i.e.
markovianity of the process $\mathbb{X}$ is not required) and that
Theorem \ref{stw2.1} is proved for such wide class of  processes.

In Section 2.3 we  show that if $\BX$ is associated with a
transient symmetric regular Dirichlet form $(\EE,D[\EE])$ on
$L^2(E;m)$, $(\BX,\PP,m)$ has the compactness property and
$\{u_n\}\subset \FF_e\cap \PP$, where $\FF_e$ is an extension of
the domain $D[\EE]$ such that the pair $(\EE,\FF_e)$ is a Hilbert
space, then the condition
\[
\sup_{n\ge 1}\EE(u_n,u_n)<\infty
\]
implies that $\{u_n\}$ has a subsequence convergent $m$-a.e.
Moreover, we prove that
\begin{align*}
&(\mathbb{X},\PP,m)\mbox{ has the compactness property iff }\\
&\qquad\qquad\qquad(\mathbb{X},\PP,\mbox{cap})\mbox{ has the
compactness property},
\end{align*}
where $\mbox{cap}$ is the capacity on $E$ determined by the form
$(\EE,D[\EE])$.

In Sections \ref{sec3} and \ref{sec4} we  define a probabilistic
solution of (\ref{eqi.1}) and give an existence result for system
(\ref{eqi.1}). The basic space in which solutions are looked for
is the space $\DD$ of measurable functions $u$ on $E$ such that
the family $ \{u(X_\tau),\tau\mbox{ is a  stopping time}\}$ is
uniformly integrable under $P_x$ for q.e. $x\in E $. We show that
$D[\EE]\subset\DD$ if $\BX$ is associated with a semi-Dirichlet
form. We call a finely continuous function $u\in\DD$ such that
$F(\cdot,u)$ is quasi-integrable a solution of (\ref{eqi.1}) if
there exists a local martingale additive functional $M$ of $\BX$
such that for $m$-a.e. $x\in E$ and every $T>0$,
\begin{align*}
u(X_t)&=u(X_{T\wedge\zeta})+\int_t^{T\wedge\zeta}
F(X_r,u(X_r))\,dr +\int_t^{T\wedge\zeta}
dA^{\mu}_r\\
&\quad+\int_t^{T\wedge\zeta} dM_r, \quad t\in
[0,T\wedge\zeta],\quad P_x\mbox{-a.s.},
\end{align*}
where  $\zeta$ is the life-time of $\BX$ and $A^{\mu}$ is the
positive co-natural additive functional associated with measure
$\mu$.

We first study probabilistic solutions to (\ref{eqi.1}) in case
$\BX$ is associated with a semi-Dirichlet form and $(\BX,\BB_1,m)$
has the compactness property. In Section \ref{sec3} we show that
if  $\mu$ is smooth and satisfies some integrabilty condition, $F$
satisfies the sign condition (\ref{eq1.3}), then there exists a
solution of (\ref{eqi.1}). We also show that if $F$ is monotone,
i.e.
\[
\langle F(x,y)-F(x,z),y-z\rangle \le 0,\quad x\in
E,\,y,z\in\mathbb{R}^N,
\]
then the probabilistic solution to (\ref{eqi.1}) is unique.

The case of general right Markov processes is considered in
Section \ref{sec4}. We show that if $\BX$ satisfies Meyer's
condition (L) then under the same hypotheses as in Section
\ref{sec3} there exists a solution to (\ref{eqi.1}). Using
(\ref{eq1.5}) one can formulate the existence result in purely
analytic terms, without relating to the  concept of the
compactness property. Namely, if the resolvent of the operator $A$
is order compact on $L^1(E;m)$, $F$ satisfies the sign condition
and the data are appropriately integrable then there exists a
solution of (\ref{eqi.1}). As a matter of fact we assume some
additional regularity condition on the semigroup $\{T_t,\, t\ge 0\}$  but we think
that it is technical and can be omitted.

In Section \ref{sec5} we give some examples of operators and
processes to which our results apply. Among others we give a
simple example of Ornstein-Uhlenbeck semigroup, i.e. semigroup
generated by differential operator of the form
\[
L\phi(x)=\frac12 \mbox{tr}(QD^2\phi(x))+\langle Ax, D\phi
(x)\rangle,
\]
which is not of variational form (or, equivalently, is not
analytic). The Ornstein-Uhlenbeck process with generator $L$ is
not associated with a Dirichlet form but satisfies Meyer's
hypothesis (L). This shows that the class of processes considered
in Section \ref{sec4} includes important processes that do not
belong to the class considered in Section \ref{sec3}.

\nsubsection{Compactness property}
 \label{sec2}

\subsubsection{Normal processes}

Let $E$ be a Radon metrizable topological space (see \cite{BB})
and $\BB(E)$ be the set of all numerical Borel measurable
functions on $E$. W adjoin an isolated point $\Delta$ to $E$ and
set $E_{\Delta}=E\cup\{\Delta\}$ (in $E_{\Delta}$ we have natural
topology in which $E$ is open). We denote by $\BB_\Delta(E)$ the
set of all numerical Borel measurable functions on $E_\Delta$. Let
$(\Omega, \mathcal{G})$ be a measurable space and
$\{X_t,\,t\in[0,\infty]\}$ be a stochastic process on $E_\Delta$
such that $X_\infty=\Delta$ and if $X_{t_0}=\Delta$ for some
$t_0\in[0,\infty]$ then $X_t=\Delta$ for $t\ge t_0$. We denote by
$\zeta$ the life-time of $X$, i.e.
\[
\zeta=\inf\{t\ge 0;\quad X_t=\Delta\}.
\]
For  $x\in E_\Delta$ let $P_x$ be a probability measure on
$(\Omega,\mathcal{G})$. Let $\{\mathcal{G}_t,\,t\in[0,\infty]\}$
be a filtration in $\mathcal{G}$ and let $\{\GG^0_t,\,t\ge 0\}$ be
a natural filtration generated by $X$. We assume that
\begin{enumerate}
\item[(a)] for every $t\ge 0$, $X_t\in\GG_t/\BB_\Delta(E)$,
\item[(b)] the mapping $E\ni x\mapsto P_x(X_t\in B)$
belongs to $\BB(E)$ for every
$t\ge 0$ and $B\in\BB(E)$,
\item[(c)] for every $x\in E_\Delta$, $P_x(X_0=x)=1$,
\item[(d)] $X$ is measurable relative to $\GG^0$, i.e. the mapping
$[0,\infty)\times\Omega\ni(t,w) \mapsto X_t(w)\in E_\Delta$ is
$\BB([0,\infty))\times\GG^0/\BB_\Delta$  measurable.
\end{enumerate}

Let $\mathbb{X}=(\Omega,\FF,\{X_t,\,t\ge 0\}, \{P_x,\, x\in E\})$.
In the whole paper for a given Borel set $B\subset E$ we denote by
\[
\sigma_B=\inf\{t>0;\, X_t\in B\},\quad D_B=\inf\{t\ge0;\, X_t\in B\},
\quad \tau_B=\inf\{t>0;\, X_t\in E\setminus B\}
\]
the hitting time, debut time and the first exist time of $B$,
respectively. By $\mathbb{X}^B=(\Omega,\FF, \{X^B_t,\, t\ge 0\},
\{P_x,\, x\in E\})$ we denote the part of $\BX$ on $B$, i.e.
\[
X^B_t= \left\{
\begin{array}{ll} X_t(\omega),& 0\le t< D_{E\setminus B}(\omega),\\
\Delta, & t\ge D_{E\setminus B}(\omega).
\end{array}
\right.
\]

Let $\BB^+(E)=\{f\in\BB(E);f(x)\ge 0,x\in E\}$ and let $\BB^r(E)$
denote the set of $u\in\BB(E)$ such that $u(x)\in\BR$ for $x\in
E$. In the whole paper we adopt the convention that $f(\Delta)=0$
for every numerical function $f$ on $E$. For every $t\ge 0$,
$\alpha\ge 0$ and $f\in\BB^+(E)$ we put
\[
p_tf(x)=E_xf(X_t),\quad R_\alpha f(x) =E_x\int_0^\infty e^{-\alpha
t}f(X_t)\,dt,\quad x\in E.
\]
By (a)--(d), $p_t:\BB^+(E)\rightarrow\BB^+(E)$,
$R_\alpha:\BB^+(E)\rightarrow \BB^+(E)$. Let
$\PP\subset\BB^{r,+}(E)$ be some  family having the following
properties
\begin{enumerate}
\item[(P1)] $(f\in\PP, g\in\BB^+(E), g\le f)\implies g\in\PP$,
\item[(P2)] $\{f_n\}\subset\PP\implies\sup_{n}f_n\in\PP$.
\end{enumerate}

Unless otherwise stated, in this section $m$ is a nonnegative
subadditive set function on $E$.

\begin{df}
(a) We say that a triple $(\BX,\PP,m)$ has the compactness
property if for every $\{u_n\}\subset\PP$ there exist a set
$\Lambda\subset(0,+\infty)$ and a subsequence $(n')\subset(n)$
such that $\sup \Lambda=+\infty$ and  for every $\alpha\in
\Lambda$ the sequence $\{R_\alpha u_{n'}\}$ is $m$-a.e. convergent
and its limit is $m$-a.e. finite.\smallskip\\
(b) We say that a pair $(\BX,\PP$) has the compactness property if
the triple ($\BX$, $\PP$, $m$) has compactness property with $m$
being the  counting measure.
\end{df}

In the sequel for given $\PP\subset\BB^{r,+}(E)$ we set
$\PP^*=\PP-\PP$.

\begin{df}
We say that a sequence $\{u_n\}\subset\PP^*$ satisfies
\begin{enumerate}
\item[(a)] condition (M$_0$) if
\[ \lim_{h\rightarrow 0^+}\sup_{n\ge 1}\sup_{t\le
h}|u_n(X_t)-u_n(x)|=0\quad \mbox{in probability }P_x\mbox{ for }
m\mbox{-a.e.}\,x\in E,
\]
\item[(b)] condition (M$_1$) if
\[\lim_{t\rightarrow 0^+}\sup_{n\ge 1}|p_tu_n(x)-u_n(x)|=0
\quad \mbox{for }m\mbox{-a.e.}\,x\in E,
\]
\item[(c)]condition (M$_2$) if $m$ is a measure and for some $p\ge1$
\[
\lim_{t\rightarrow 0^+}\sup_{n\ge 1}\|p_tu_n-u_n\|_{L^p(E;m)}=0.
\]
\end{enumerate}
\end{df}

\begin{uw}
It is clear that if $\{u_n\}$ satisfies some  integrability
conditions and $m$ is a $\sigma$-finite measure then (M$_0$) 
implies (M$_1$) and (M$_1$) implies (M$_2$).
\end{uw}

\begin{tw}
\label{stw2.1} Assume that $(\BX,\PP,m)$ has the compactness
property. If $\{u_n\}\subset\PP^*$ satisfies $\mbox{\rm(M$_1$})$
then there exists a subsequence $(n')\subset(n)$ such that
$\{u_{n'}\}$ is $m$-a.e. convergent and its limit is $m$-a.e.
finite.
\end{tw}
\begin{dow}
Let  $\Lambda\subset(0,+\infty)$ be a countable set such that
$\sup \Lambda=+\infty$ and let $(n')\subset(n)$ be a subsequence
such that for every $\alpha\in \Lambda$, $\{R_\alpha u_{n'}\}$ is
$m$-a.e. convergent and its limit is finite $m$-a.e. Let $A\subset
E$ be a set of those $x\in E$ for which $\lim_{n'}R_\alpha
u_{n'}(x)$ does not exist or exists and is infinite for some
$\alpha\in \Lambda$. It is clear that $m(A)=0$. Let $B$ be the set
of those $x\in E$ for which condition (M$_1$) does not hold. We
put $w=\sup_n|u_n|$. By (P1) and (P2), $u^+_n,u_n^-\in\PP$ and
$\sup_nu_n^+,\sup_nu_n^-\in\PP$. Since
$w\le\sup_nu_n^++\sup_nu_n^-$, we see that $w(x)<\infty$ for $x\in
E$ and that without loss of generality we may assume that
$R_\alpha w(x)<\infty$ for $m$-a.e. $x\in E$ and every $\alpha\in
\Lambda$. Let $C\subset E$ be the set of those $x\in E$ for which
$R_\alpha w(x)=+\infty$ for some $\alpha\in \Lambda$. Let $N=A\cup
B\cup C$. It is clear that $m(N)=0$. Let $x\in E\setminus N$. Then
\begin{equation}
\label{eq2.1}
|\alpha R_\alpha u_n(x)-u_n(x)|
\le\alpha\int_0^\infty e^{-\alpha t}|p_tu_n(x)-u_n(x)|\,dt.
\end{equation}
Let us fix $\varepsilon>0$ and let $\theta_x^\varepsilon>0$,
$n_x^\varepsilon\in\BN$ be such that
\begin{equation}
\label{eq2.2}
\sup_{t\le\theta^\varepsilon_x}|p_tu_n(x)-u_n(x)|<\frac
\varepsilon 2, \quad n\ge n^\varepsilon_x.
\end{equation}
Then
\begin{align*}
&\alpha\int_0^\infty e^{-\alpha t}|p_tu_n(x)-u_n(x)|\,dt\\
&\qquad \le\alpha\int_0^{\theta_x^\varepsilon} e^{-\alpha t}
|p_tu_n(x)-u_n(x)|\,dt+ \alpha\int_{\theta_x^\varepsilon}^\infty
e^{-\alpha t}|p_tu_n(x)-u_n(x)|\,dt\\
&\qquad\le
\alpha\frac{\varepsilon}{2}\int^{\theta_x^\varepsilon}_0
e^{-\alpha t}\,dt +\alpha\int_{\theta_x^\varepsilon}^\infty
e^{-\alpha t}p_tw(x)\,dt +\alpha
w(x)\int_{\theta_x^\varepsilon}^\infty e^{-\alpha t}\,dt\\
&\qquad\le \frac\varepsilon
2(1-e^{-\alpha\theta^\varepsilon_x})+w(x)e^{-\alpha\theta^\varepsilon_x}+
\alpha\int_{\theta_x^\varepsilon}^\infty e^{-\alpha t}p_tw(x)\,dt.
\end{align*}
Let $\alpha_0\in\Lambda$ be such that
$\alpha_0>1/\theta_x^\varepsilon$. Then we have
\[
\alpha e^{-\alpha t}p_tw(x)\le\alpha_0e^{-\alpha_0t}p_tw(x),\quad
t\ge\theta^\varepsilon_x,\quad\alpha\ge\alpha_0,\quad\alpha\in\Lambda.
\]
Therefore  there exists $\alpha_x^\varepsilon\in \Lambda$ such
that
\begin{equation}
\label{eq2.3}
|\alpha R_\alpha u_n(x)-u_n(x)|<\varepsilon,\quad n\ge n^\varepsilon_x,
\quad\alpha\in \Lambda,\quad
\alpha\ge\alpha^\varepsilon_x.
\end{equation}
Write $\beta_x=\limsup_{n'}u_{n'}(x)-\liminf_{n'} u_{n'}(x)$. Then
there exists a subsequence $(n'_k)\subset(n')$ such that
\begin{equation}
\label{eq2.4}
|u_{n'_{k+1}}(x)-u_{n'_k}(x)|>\frac{\beta_x}{2},\quad k\ge1.
\end{equation}
On the other hand, by (\ref{eq2.3}), for $n'_k\ge n_x^\varepsilon$
and $\alpha\in \Lambda$ such that $\alpha\ge\alpha_x^\varepsilon$,
\begin{align}
\label{eq2.5} \nonumber |u_{n'_{k+1}}(x)-u_{n'_k}(x)| &\le
|u_{n'_{k+1}}(x)-\alpha R_\alpha u_{n'_{k+1}}(x)|+|\alpha R_\alpha
u_{n'_{k+1}}(x)-\alpha R_\alpha u_{n'_k}(x)|\nonumber\\& \quad
+|\alpha R_\alpha u_{n'_k}(x)-u_{n'_k}(x)|\nonumber\\
&\le 2\varepsilon+|\alpha R_\alpha u_{n'_{k+1}}(x)- \alpha
R_\alpha u_{n'_k}(x)|.
\end{align}
Put $\varepsilon=\beta_x/9$. By the compactness property  of the
triple $(\mathbb{X},\PP,m)$ there exists $N(\varepsilon,x,
\alpha_x^\varepsilon)\in\BN$ such that
\begin{equation}
\label{eq2.6} |\alpha_x^\varepsilon
R_{\alpha_x^\varepsilon}u_{n'_{k+1}}(x)- \alpha_x^\varepsilon
R_{\alpha_x^\varepsilon}u_{n'_k}(x)|<\frac{\beta_x}{9}
\end{equation}
for $n'_k\ge N(\varepsilon,x,\alpha_x^\varepsilon)$. By
(\ref{eq2.5}) and (\ref{eq2.6}),
\[
|u_{n'_{k+1}}(x)-u_{n'_{k}}(x)|\le\frac{\beta_x}{3}
\]
for
$n'_k\ge\max\{n_x^\varepsilon,N(\varepsilon,x,\alpha_x^\varepsilon)\}$,
which contradicts (\ref{eq2.4}) and  proves the theorem.
\end{dow}

\begin{wn}
If $(\BX,\PP,m)$ has the compactness property and $\{u_n(X)\}$ is
a sequence of c\`adl\`ag processes on some interval $[0,T]$ tight
in the Skorokhod topology  $J_1$ under the measure $P_x$ for
$m$-a.e. $x\in E$ then $\{u_n\}$ has a subsequence convergent
$m$-a.e.
\end{wn}
\begin{proof}
If $\{u_n(X)\}$ is tight under $P_x$ for $m$-a.e. $x\in E$ then
condition (M$_0$) is satisfied for $m$-a.e. $x\in E$ and
$\sup_{n\ge 1}|u_n(x)|$ is finite $m$-a.e. Of course the same is
true for $\{u^k_n(X)\}$ for every $k\ge 0$, where
$u^k_n=T_k(u_n)$. Observe that if $\{u^k_n\}$ satisfies (M$_0$)
then it satisfies (M$_1$). Therefore from Theorem \ref{stw2.1} it
follows  that $\{u^k_n\}$ converges $m$-a.e. up to a subsequence
for every $k\ge1$. From this we easily deduce that there exists a
subsequence $(n')\subset(n)$ such that $\{u_{n'}\}$ converges
$m$-a.e.
\end{proof}

\subsubsection{Right Markov processes}

Let us recall that $\BX$ satisfies Meyer's hypothesis (L) if there
exists a $\sigma$-finite Borel measure $m$ on $E$ such that
$R_\alpha(x,dy)\ll m$ for every $x\in E$ and some (and hence
every) $\alpha>0$.

The measure $m$ of the above definition will be called a reference
measure for the process $\BX$ or a reference measure for  the
resolvent $\{R_{\alpha},\alpha>0\}$.

Let $\BB_1=\{u\in\BB^{+}(E);\,u(x)\le1,\,x\in E\}$.

\begin{stw}
\label{stw2.2} Assume that $\BX$ is a right Markov process. Then
$(\BX,\BB_1)$ has the compactness property iff $\BX$ satisfies
Meyer's hypothesis \mbox{\rm(L)}.
\end{stw}
\begin{dow}
Assume that $\BX$ has a reference measure $m$. Then for every
$\alpha>0$ there exists a measurable function $r_\alpha:E\times
E\rightarrow\BR_+$ such that
\[
R_\alpha(x,dy)=r_\alpha(x,y)\,m(dy)\quad\mbox{for every }x\in E.
\]
Since $\alpha R_\alpha\le 1$,
\[
\int_E r_\alpha(x,y)m(dy)=R_\alpha 1\le\alpha,\quad x\in E,
\]
which implies that $r_\alpha(x,\cdot)\in L^1(E;m)$ for every $x\in
E$. If $\{u_n\}\subset\BB_1(E)$ then of course $\{u_n\}\subset
L^\infty(E;m)$ and $\sup_n\|u_n\|_\infty<\infty$, so there exists
a subsequence $(n')\subset(n)$ such that $\{u_{n'}\}$ converges
weakly${}^*$ in $L^\infty$ to some $w\in L^\infty(E;m)$, i.e. for
every $v\in L^1(E;m)$,
\[\int_Eu_{n'}\cdot v\,dm\rightarrow\int_Ew\cdot v\,dm.
\]
In particular, for every $x\in E$,
\[
R_\alpha u_{n'}(x)=\int_Er_\alpha(x,y)u_{n'}(y)m(dy)\rightarrow
\int_Er_\alpha(x,y)w(y)m(dy)=R_\alpha w(x),
\]
which shows that $(\BX,\BB_1)$ has the compactness property.

Now assume that $(\BX,\BB_1)$ has the compactness property. From
the resolvent identity it is clear that if $R_\alpha$  has a
reference measure for some $\alpha>0$ then $\BX$ has a reference
measure. Consequently, if $\BX$ does not have a reference measure
then  for every $\alpha>0$ the resolvent $R_\alpha$ does not have
a reference measure. In \cite{WeisWerner} it is proved that if
$R_\alpha$ does not have a reference measure then there exists a
compact perfect set $K\subset E$ such that $\mathbf{1}_K\cdot
R_\alpha:\BB_b(E)\rightarrow\BB_b(K)$ is surjective. Moreover,
from the proof it follows that there exists $\gamma>0$ such that
for every $g\in\BB_b^+(K)$ such that $\sup_{x\in E}|g(x)|\le c$
for some $c>0$ there exists $f\in\BB_b^+(E)$ such that
$(\mathbf{1}_KR_\alpha)f=g$ and $|f(x)|\le\gamma c$ for $x\in E$.
Since $K$ is uncountable, there exists a sequence
$\{v_n\}\subset\BB_b^+(K)$ such that $|v_n(x)|\le1/\gamma$ for
$x\in E$, $n\ge 1$ and $\{v_n\}$ has no subsequence converging
pointwise. Thanks to the properties of the operator
$\mathbf{1}_KR_\alpha$, for every $n\ge 1$ there exists
$u_n\in\BB_1^+(E)$ such that $\mathbf{1}_KR_\alpha u_n=v_n$. This
implies that there exists a sequence $\{u_n\}\subset\BB_1^+(E)$
such that $\{R_\alpha u_n\}$ has no subsequence converging
pointwise.
\end{dow}

\begin{uw}
Let $\BX$ be a Markov process and $m$ be its excessive measure. If
for some $\alpha>0$, $R_\alpha$ carries a family $\PP$ in a
relatively compact set in the topology of $m$-a.e. convergence
then $R_\beta$ has the same property for every $\beta>\alpha$. To
see this, let us suppose that $\beta>\alpha$ and
$\{u_n\}\subset\PP^*$. Then there exists a subsequence
$(n')\subset(n)$ such that $\{R_\alpha u_{n'}\}$ is $m$-a.s.
convergent and its limit is finite $m$-a.e. By the resolvent
identity,
\begin{equation}
\label{eq2.11}
R_\beta u_{n'}=R_\alpha u_{n'}+(\alpha-\beta)R_\beta(R_\alpha u_{n'}).
\end{equation}
Set $v=\lim_{n'}R_\alpha u_{n'}$ and $w=\sup_{n'}|u_{n'}|$. By
(P1) and (P2), $w\in\PP$. Hence  $R_\alpha w<\infty$, $m$-a.e.,
and consequently, $R_\beta w<\infty$ $m$-a.e. for every
$\beta\ge\alpha$. Let $g$ be a positive Borel measurable function
on $E$ such that $\int_ER_\alpha w\cdot g\,dm<\infty$. Since $m$
is an excessive measure, applying the Lebesgue dominated
convergence theorem yields
\[
\beta\int_ER_\beta|R_\alpha u_{n'}-v|g\,dm\le
\int_E|R_\alpha u_{n'}-v|g\,dm\rightarrow 0
\]
as $n'\rightarrow\infty$. From this we conclude that there exists
a subsequence $(n'')\subset(n')$ such that $R_\beta(R_\alpha
u_{n''})$ is $m$-a.e. convergent and its limit is finite $m$-a.e.
This when combined with (\ref{eq2.11}) implies the convergence of
$\{R_\beta u_{n''}\}$. Therefore if $\mathbb{X}$ is a Markov
process then the compactness of a triple $(\mathbb{X},\PP,m)$ is
equivalent to saying that for some $\alpha>0$ the operator
$R_\alpha$ carries the family $\PP$ in a relatively compact set in
the topology of $m$-a.e. convergence.
\end{uw}

\subsubsection{Hunt processes associated with Dirichlet forms}

In the rest of this section $\BX$ is a Hunt process associated
with a regular semi-Dirichlet form on $L^2(E;m)$.  Let us recall
that a semi-Dirichlet form on $L^2(E;m)$ is a bilinear form
\[
\mathcal{E}: D[\EE]\times D[\EE]\rightarrow \mathbb{R}
\]
defined on a dense linear subspace $D[\EE]$ of $L^2(E;m)$
satisfying the following conditions
\begin{enumerate}
\item[\rm{(a)}] there exists $\alpha_0\ge 0$ such that $\EE_\alpha(u,u)\ge 0$
for every
$u\in D[\EE]$ and $\alpha\ge\alpha_0,$
\item[\rm{(b)}]there exists $K>0$ such that
$|\EE(u,v)|\le K (\EE_{\alpha_0}(u,u))^{1/2}
(\EE_{\alpha_0}(v,v))^{1/2}$ for every $u,v\in D[\EE]$,
\item[\rm{(c)}] $D[\EE]$ equipped with the inner product
$\EE^{(s)}_{\alpha_0}(\cdot,\cdot)$, where
$\EE_{\alpha}^{(s)}(u,v)=\frac12(\EE_{\alpha}(u,v)+\EE_{\alpha}(v,u))$,
is a Hilbert space,
\item[\rm{(d)}] for every $u\in D[\EE]$ and $k\ge 0$, $u\wedge k\in D[\EE]$
and $\EE(u\wedge k,u\wedge k)\le \EE(u\wedge k,u)$.
\end{enumerate}

A semi-Dirichlet form $(\EE,D[\EE])$ is called regular if there
exists a set $\mathcal{C}\subset C_0(E)\cap D[\EE]$ ($C_0(E)$ is
the set of all continuous functions on $E$ with compact suppport)
such that $\mathcal{C}$ is dense in $D[\EE]$ in the norm
determined by $\EE_{\alpha_0}$ and in $C_0(E)$ in the norm of
uniform convergence.

It is well known that with every semi-Dirichlet form
$(\EE,D[\EE])$ one can associate uniquely a Hunt process
$\mathbb{X}$ (see \cite[Section 3.3]{Oshima}).

A semi-Dirichlet form $(\EE,D[\EE])$ is called positive if (a) is
satisfied with $\alpha_0=0$ and  is called transient if the
associated Hunt process $\mathbb{X}$ is transient, i.e. there
exists a strictly positive Borel measurable function $f$ on $E$
such that $Rf$ is finite $m$-a.e. It is known that if a
semi-Dirichlet form $(\EE,D[\EE])$ is transient and positive then
there exists an extension $\FF_e$ of the domain $D[\EE]$ such
that $(\FF_e,\EE^{(s)}(\cdot,\cdot))$ is a Hilbert space.

By $\mbox{cap}$ we denote the capacity on subsets of $E$ naturally
associated with $(\EE,D[\EE])$ (see \cite[Section 2.1]{Oshima}).
We say that some property holds quasi everywhere (q.e. for short)
if it holds except for a set $N\subset E$ such that
$\mbox{cap}(N)=0$.

We say that an increasing sequence  $\{F_n\}$ of closed subsets of
$E$ is a nest if for every compact $K\subset E$,
$\mbox{cap}(K\setminus F_n)\rightarrow 0$ as $n\rightarrow
\infty.$

We say that a Borel measure $\mu$ on $E$ is smooth if it charges
no set of zero capacity and there exists a nest $\{F_n\}$ such
that $|\mu|(F_n)<\infty,\, n\ge 1$.

It is well known (see \cite[Section 4.1]{Oshima}) that for every
smooth measure $\mu$ there exists a unique continuous additive
functional $A^\mu$ of $\mathbb{X}$ in the Revuz duality with
$\mu$.

In the whole paper  for a positive smooth measure $\mu$ and $\alpha\ge 0$ we write
\[
(R_\alpha\mu)(x)=E_x\int_0^{\zeta}e^{-\alpha r}dA^\mu_r.
\]
Observe that if $f\in \mathcal{B}^{+}(E)$  then $R_\alpha(f\cdot
m)=R_\alpha f$, where $R_\alpha$ is defined by (\ref{eq1.4}). We
also write $R=R_0$.

\begin{lm}
\label{lm1.1} Let $\{u_n\}\subset\BB(E)$ be such that $R_\alpha
w<\infty$ $m$-a.e., where $w=\sup_n|u_n|$.  If $\{R_\alpha u_n\}$
is convergent $m$-a.e. then there exists a subsequence
$(n')\subset(n)$ such that $\{R_\alpha u_{n'}\}$ is convergent
q.e.
\end{lm}
\begin{dow}
Let $\{g_k\}$ be a sequence of Borel measurable functions on $E$
such that $0\le g_k(x)\le1$, $g_k(x)\nearrow 1$ for $x\in E$ and
$g_k\cdot w\in L^2(E;m)$ for every $k\ge 1$. Write $u_n^k=g_ku_n$,
$v^k_n=R_\alpha(u^k_n)$, $v_n=R_\alpha(u_n)$. Then
\begin{equation}
\label{eq2.7}
|v_n^k(x)-v_n(x)|\le R_\alpha(w|1-g_k|)(x),\quad x\in E.
\end{equation}
Let $B=\{\inf_kR_\alpha w|1-g_k|>0\}$ and let $K$ be a compact set
such that $K\subset B$. Then
\begin{align*}
P_m(\sigma_K<\infty)&=P_m(\sigma_K<\infty, \inf_{k\ge 1}E_x\left(
\int_{\sigma_K}^\infty w|1-g_k|(X_r)\,dr|\FF_{\sigma_K}\right)>0)
\\&\quad \le P_m(\sigma_K<\infty,\inf_{k\ge 1}E_x\left(\int_0^\infty
w|1-g_k|(X_r)\,dr|\FF_{\sigma_K}\right)>0)=0
\end{align*}
since $\inf_kR_\alpha(w|1-g_k|)(x)=0$ for $m$-a.e. $x\in E$. From
this and (\ref{eq2.7}) we conclude that
\begin{equation}
\label{eq2.8} \lim_{k\rightarrow\infty}\sup_{n\ge
1}|v_n^k(x)-v_n(x)|=0\quad\mbox{for q.e. }x\in E.
\end{equation}
Therefore to prove the lemma it suffices to show  that for every
$k\ge 0$, $\{v_n^k\}$ is convergent q.e. But this follows
immediately from the inequality
\begin{align*}
\mathcal{E}_\alpha(R_\alpha u_n^k-R_\alpha u_m^k,\, R_\alpha
u_n^k-R_\alpha u_m^k)&= (u_n^k-u_m^k,\,
R_\alpha u_n^k-R_\alpha u_m^k)_{L^2(E;m)}\\
&\le2\|g_k\cdot w\|_{L^2(E;m)}\cdot\|v_n^k-v_m^k\|_{L^2(E;m)}
\end{align*}
and \cite[Theorem 2.2.5]{Oshima}.
\end{dow}
\medskip

\begin{wn}
\label{wn2.7} A triple $(\mathbb{X},\PP,m)$ has the compactness
property iff $(\mathbb{X},\PP,\mbox{\rm cap})$ has the compactness
property.
\end{wn}

For $B\in \BB(E)$ set
\[
\PP(B)=\{u\in\BB(E);u(x)=0,x\in E\setminus B\}
\]
and for $\alpha\ge 0$ and $u\in\BB^+(E)$ set
\[R_{\alpha}^{E\setminus B}u(x)=E_x\int_0^{\sigma_B}e^{-\alpha t}u(X_t)\,dt,
\quad H_B^\alpha u(x)=E_xe^{-\alpha\sigma_B}u(X_{\sigma_B}).
\]

\begin{stw}
\label{stw.2} Let $B\in\BB(E)$. If $(\BX,\PP,m)$ has the
compactness property then $(\BX^B,\PP(B),m)$ has the compactness
property.
\end{stw}
\begin{dow}
Let $\{u_n\}\subset\PP^+(B)$. By the assumption there exists  a
set $\Lambda\subset(0,+\infty)$ such that $\sup \Lambda=+\infty$
and a subsequence $(n')\subset(n)$ such that for every $\alpha\in
\Lambda$ the sequence $\{R_\alpha u_{n'}\}$ is convergent $m$-a.e.
and its limit is finite $m$-a.e. By Dynkin's formula,
\[
R_{\alpha} u_{n'}=R_\alpha^Bu_{n'}+H_{E\setminus
B}^\alpha(R_\alpha u_{n'}), \quad m\mbox{-a.e.}
\]
Therefore it suffices to show that up to a subsequence,
$\{H^\alpha_{E\setminus B}(R_\alpha u_{n'})\}$ is $m$-a.e.
convergent and its limit is finite $m$-a.e. But this follows
immediately from Lemma \ref{lm1.1}, because q.e. convergence of
$\{R_\alpha u_{n'}\}$ implies that $\{e^{-\alpha\tau_B}R_\alpha
u_{n'}(X_{\tau_{B}})\}$ is convergent $P_x$-a.s. for $m$-a.e.
$x\in E$, moreover we have $|R_\alpha u_n|\le R_\alpha w$, $m$-a.e., where
$w=\sup_n|u_n|\in\PP$. Therefore we can apply the Lebesgue
dominated convergence theorem to sequence $\{e^{-\alpha\tau_B}R_\alpha
u_{n'}(X_{\tau_{B}})\}$ , because $\{|R_\alpha u_n|>R_\alpha
w\}$ as a finely open $m$-negligible set is exceptional,  which in turn implies that $|R_\alpha u_n|\le
R_\alpha w$ q.e., hence that $|(R_\alpha
u_n)|(X_{\tau_{B}})\le(R_\alpha w)(X_{\tau_{B}})$, $P_x$-a.s. for
$m$-a.e. $x\in E$ and $E_xR_\alpha
w(X_{\tau_{B}})=E_x\int_{\tau_{B}}^\infty e^{-\alpha r}w(X_r)\,dr
\le E_x\int_0^\infty e^{-\alpha r}w(X_r)\,dr=w(x)<\infty$ for
every $x\in E$.
\end{dow}

\begin{uw}
\label{uw2.9} Observe that the assertion of Proposition
\ref{stw.2} holds true if we replace the process $\BX^{B}$ killed
outside a Borel set $B$ by the process $\BX^{A}$ killed with rate
$-dL_t$, where $L_t=e^{-A_t}$ for some positive continuous
additive functional $A$ of $\BX$ (for notation see \cite[Theorem
A.2.11]{Fukushima}). To see this it suffices to repeat the proof
of Proposition \ref{stw.2} with $\tau_B$ replaced by the stopping
time
\[
\zeta_A=\inf\{t<\zeta;\, A_t\ge Z\},
\]
where $Z$ is a random variable of exponential distribution with
mean $1$ independent of $\BX$ and satisfying
$Z(\theta_s(\omega))=(Z(\omega)-s)\vee0$.
\end{uw}

Let us recall that a Markov process $\BX^0$ on $E_0\in\BB(E)$ is
called a subprocess of $\BX$ if its semigroup $\{p_t^0,t\ge 0\}$
extends naturally to $E$ subordinate to $\{p_t,t\ge 0\}$, i.e. for
every $f\in\BB^+(E)$ and $t\ge 0$,
\[
\overline{p^0_t}f\le p_tf,
\]
where $\overline{p^0_t}f(x)=p^0_tf_{|_{E_0}}(x)$ for $x\in E_0$
and $\overline{p^0_t}f(x)=0$ for $x\in E\setminus E_0$.

\begin{wn}
\label{wn2.222} Let $\BX^0$ be a subprocess of $\BX$. If
$(\BX,\PP,m)$ has the compactness property then $(\BX^0,\PP,m)$
has the compactness property.
\end{wn}
\begin{dow}
By \cite[page 103]{BG}, $\BX^0$ is equivalent to the process $\BX$
killed with rate $-dL_t$, where  $L_t=e^{-A_t}$ for some PCAF $A$
of $\BX$. Therefore the desired result follows from Remark
\ref{uw2.9}.
\end{dow}

\begin{stw}
\label{stw2.3} Assume that $\EE$ is positive and $(\BX,\PP,m)$ has
the compactness property. If $\{u_n\}\subset\FF_e\cap\PP$ and
\begin{equation}
\label{eq2.100} \sup_{n\ge1}\mathcal{E}(u_n,u_n)<\infty
\end{equation}
then there exists a subsequence $(n')\subset(n)$ such that
$\{u_{n'}\}$  is $m$-a.e. convergent and its limit is $m$-a.e.
finite.
\end{stw}
\begin{dow}
Let $\eta\in D[\mathcal{E}]$ be such that  $\eta>0$, $m$-a.e. and
$\|\eta\|_\infty<\infty$. By known properties of Dirichlet forms,
\[
\mathcal{E}(T_k(w)\eta,T_k(w)\eta)\le \|\eta\|_\infty
\mathcal{E}(u_n,u_n)+k\mathcal{E}(\eta,\eta).
\]
Hence
\begin{equation}
\label{eq2.9}
\sup_n\mathcal{E}(T_k(u_n)\eta,\,T_k(u_n)\eta)<\infty.
\end{equation}
We can assume that $u_n\ge 0$, $m$-a.e.  for every $n\ge 1$,
because  from (P1) it follows that $u_n^{+}\in\PP$ and it is known
that $u_n^{+}\in\FF_e$ a and $
\mathcal{E}(u_n^{+},u_n^{+})\le\mathcal{E}(u_n,u_n)$.  Under the
assumption of nonnegativity of $u_n$, $u^k_n\equiv
T_k(u_n)\cdot\eta\in\FF_e\cap\PP\cap L^2(E;m)$. By an elementar calculus,
\[
\|\alpha R_\alpha u_n^k-u_n^k\|_{L^2(E;m)}\le\alpha^{-1}
\mathcal{E}(u_n^k,u_n^k),
\]
which when combined with (\ref{eq2.9}) gives
\begin{equation}
\label{eq2.10} \|\alpha R_{\alpha}
u_n^k-u_n^k\|_{L^2(E;m)}\le\alpha^{-1}c(k,\eta).
\end{equation}
By the assumption there exists a subsequence $(n')\subset(n)$ and
a subset $\Lambda\subset(0,\infty)$ such that $\sup
\Lambda=+\infty$ and $\{\alpha R_\alpha u^k_{n'}\}$ is convergent
in $L^2(E;m)$ for every $\alpha\in \Lambda$. This and
(\ref{eq2.10}) imply that there exists a further subsequence
$(n'')\subset(n')$ such that $\{u^k_{n''}\}$ is convergent in
$L^2(E;m)$. From this it follows easily that $\{u_{n'''}\}$ is
convergent $m$-a.e. for some further  subsequence
$(n''')\subset(n'')$.
\end{dow}

\begin{uw}
If $\{u_n\}\subset \FF_e$ satisfies (\ref{eq2.100}) then by the
calculations in the proof of \cite[Lemma 1.5.4.]{Fukushima} it
satisfies condition (M$_2$).
\end{uw}

In the sequel by $\mathcal{T}$ we denote the set of all stopping times
to given filtration $\FF$.
\begin{df}
We say that a Borel measurable function $u$ on $E$ is of class
(FD) if for $m$-a.e. $x\in E$ the family
$\{u(X_{\tau}),\,\tau\in\mathcal{T}\}$ is uniformly integrable
under the measure $P_x$.
\end{df}

By $\mathbf{D}$ we denote the set of all Borel measurable
functions on $E$ of class (FD).

\begin{uw}
\label{uw.d} (i) Observe that  $D[\EE]\subset \mathbf{D}$. Indeed,
each positive $u\in D[\EE]$ is majorized by the $\alpha$-potential
$e^{\alpha}_u$ (the smallest $\alpha$-potential majorizing $u$)
and $e^\alpha_u=U_\alpha\mu$ for some measure  $\mu$ of finite energy integral
(see \cite[Theorem 2.3.1]{Oshima}).
Therefore by \cite[Theorem 4.1.10]{Oshima},
\[
e^{\alpha}_u(x)=E_x\int_0^{\infty}e^{-\alpha t}\,dA^{\mu}_t.
\]
From the above formula we easily deduce that $e^\alpha_u\in
\mathbf{D}$ which implies that $u\in\mathbf{D}$. Since $u^+,
u^-\in D[\EE]$ if $u\in D[\EE]$,  we get the result.
\smallskip\\
(ii) If we assume additionally that $\EE$ is positive and
transient then in the same manner as in (i) we can show that
$\FF_e\subset\mathbf{D}$.
\end{uw}

For $\alpha \ge 0$ and $\rho\in\mathcal{B}(E)$ such  that $\rho>0$
let us define the space
\[
\mathbf{D}_\alpha=\{u\in\mathbf{D};\, \|u\|_\alpha<\infty\},
\]
where
\[
\|u\|_\alpha
=\int_E \sup_{\tau\in\mathcal{T}}E_xe^{-\alpha\tau}|u(X_\tau)|\,m(dx).
\]

In the sequel for a given $v\in\mathcal{B}^+(E)$ we write
\[
[0,v]=\{u\in\mathcal{B}(E); \,\, 0\le u\le v\}.
\]

\begin{stw}
Let $(\EE,D[\EE])$ be a regular symmetric Dirichlet form and $\BX$
be a  Hunt process associated with $(\EE,D[\EE])$. Then
$(\BX,\BB_1,m)$ has the compactness property  iff for every
$\alpha>0$ the mapping $R_\alpha:\DD_0\rightarrow \DD_\alpha$ is
order compact.
\end{stw}
\begin{dow}
Assume that $(\BX,\BB_1,m)$ has the compactness property. Let $v\in
\DD_0$ and $\{u_n\}\subset[0,v]$. Let $\{g_k\}$ be a sequence of
positive Borel measurable functions on $E$  such that $g_k\nearrow
1$ as $k\rightarrow\infty$ and $g_k\cdot v\in L^2(E;m)$. Put
$u_n^k=g_kT_k(u_n)$, $v_n^k=R_\alpha(u_n^k)$, $v_n=R_\alpha(u_n)$.
Then
\begin{align}
\label{eqa.2} \nonumber \EE_\alpha(R_\alpha u_n^k-R_\alpha
u_m^k,R_\alpha u_n^k-R_\alpha u_m^k)&= (u_n^k-u_m^k;\,R_\alpha
u_n^k-R_\alpha u_m^k)\\
&\le 2\|g_k\cdot v\|_{L^2(E;m)}\cdot\|v_n^k-v_m^k\|_{L^2(E;m)}.
\end{align}
By the assumption, without loss of generality we may assume that
for every $k\in\BN$ the sequence $\{v_n^k\}$ is $m$-a.e.
convergent as $n\rightarrow\infty$. Since $v_n^k\le
R_\alpha(g_k\cdot v)\in L^2(E;m)$, $\{v_n^k\}$ converges in
$L^2(E;m)$, and hence,  by (\ref{eqa.2}),  in $\EE_\alpha$. By
\cite[Lemma 5.1.1]{Fukushima} this implies that there exists a
subsequence (still denoted by $n$) such that for q.e. $x\in E$,
\[
\lim_{n,m\rightarrow\infty}E_x\sup_{t\ge 0}e^{-\alpha
t}|v_n^k(X_t)-v_m^k(X_t)|=0.
\]
Hence
\[
\lim_{n,m\rightarrow\infty}\sup_{\tau\in \mathcal{T}}
E_xe^{-\alpha \tau}|v_n^k(X_\tau)-v_m^k(X_\tau)|=0
\]
for q.e. $x\in E$. By the Lebesgue dominated convergence theorem,
$\|v_n^k-v_m^k\|_\alpha\rightarrow0$ as $n,m\rightarrow\infty$, so
it is enough to show that $\|v_n^k-v_n\|_\alpha\le C(k)$ for some
$C(Kk$ such that $C(k)\rightarrow 0$ as $k\rightarrow\infty$. To
this end, let us observe that
\begin{align*}
\|v_n^k-v_n\|_\alpha&\le\int_E\sup_{\tau\in\mathcal{T}}E_xe^{-\alpha\tau}
|R_\alpha u_n^k(X_\tau)-R_\alpha u_n(X_\tau)|\,m(dx)\\&
\le\int_E\sup_{\tau\in\mathcal{T}}E_x\Big(e^{-\alpha\tau}
E_x\Big(\int_\tau^\infty e^{-\alpha(r-\tau)} |u_n^k(X_r)-
u_n(X_r)|\,dr|\FF_\tau\Big)\Big)\,m(dx)\\
&\le\int_EE_x\int_0^\infty e^{-\alpha r}
|u_n^k(X_r)-u_n(X_r)|\,dr\,m(dx)\\
&\le \int_EE_x\int_0^\infty
e^{-\alpha r}|g_kT_k(u_n)-u_n|(X_r)\,dr\,m(dx)\\
&\le\tau\int_EE_x\int_0^\infty e^{-\alpha
r}\mathbf{1}_{\{v>k\}}v(X_r)\,m(dx)\\
&\quad +\int_EE_x\int_0^\infty e^{-\alpha r}
(v|g_k-1|)(X_r)\,dr\,m(dx)\equiv C(k).
\end{align*}
Since $v\in \DD_0$, both integrals on the right-hand side of last
inequality are finite. Therefore by the Lebesgue dominated
convergence theorem, $C(k)\rightarrow 0$ as $k\rightarrow\infty$,
which shows the ``if" part. Now assume that
$R_\alpha:\mathbf{D}_0\rightarrow \mathbf{D}_\alpha$ is order
compact. Let $\{u_n\}\subset\BB^+(E)$be such that $u_n(x)\le 1$
for $x\in E$. It is clear that $1\in \mathbf{D}_0$ and
$\{u_n\}\subset[0,1]$, so by order compactness of
$R_\alpha:\mathbf{D}_0\rightarrow \mathbf{D}_\alpha$ it follows
that  there exists a subsequence (still denoted by $n$) such that
\[
\lim_{n,m\rightarrow\infty}\|R_\alpha u_n-R_\alpha u_m\|_\alpha=0.
\]
In particular $\|R_\alpha u_n-R_\alpha u_m\|_{L^1(E;\rho\cdot
m)}\rightarrow 0$ as $n,m\rightarrow\infty$ from which we conclude
that $(\mathbb{X},\mathcal{B}_1,m)$ has the compactness property.
\end{dow}

\nsubsection{Elliptic systems with measure data on Dirichlet
space} \label{sec3}

In this section we assume that $(\EE,D[\EE])$ is a transient
regular semi-Dirichlet form on $L^2(E;m)$. By $\BX$ we denote a
Hunt process associated with $(\EE,D[\EE])$.

In the sequel we adopt the convention that an $N$-dimensional
process $Y$ or function $u$ has some property defined for
one-dimensional processes or functions (for instance $Y$ is a MAF
or CAF of $\BX$, $u$ is of class (FD) etc.) if its each coordinate
has this property.

Let $F:E\times\BR^N\rightarrow\BR^N$ be a measurable function and
$\mu=(\mu_1,\dots,\mu_N)$ be a Borel measure on $E$ such that
\begin{enumerate}
\item[(H1)] $\mu_i$ is a smooth measure such that $R|\mu_i|<\infty$ q.e.,
\item[(H2)] for every $r\ge 0$ the mapping $x\mapsto\sup_{|y|\le
r}|F(x,y)|$ belongs to $qL^1(E;m)$,
\item[(H3)] for every $x\in E$ the mapping $y\mapsto F(x,y)$ is continuous,
\item[(H4)] there exists a non-negative function $G$ such that
$RG<\infty$ q.e. and for every $x\in E$ and $y\in\BR^N$,
\[
\langle F(x,y),y\rangle\le G(x)|y|.
\]
\end{enumerate}

\begin{df}
We say that a Borel measurable function $f$ on $E$ is
quasi-integrable if for q.e. $x\in E$,
\[
P_x\Big(\int_0^{\zeta\wedge T}|f(X_r)|\,dr<\infty,\,T>0\Big)=1.
\]
By $qL^1(E;m)$ we denote that set of all quasi-integrable functions on $E$.
\end{df}

\begin{uw}
In the literature one can find another definition of
quasi-integrability which we call here quasi-integrability in the
analytic sense.  According to this definition a measurable
function $f$ on $E$ is quasi-integrable if for every
$\varepsilon>0$ there exists an open set $U_{\varepsilon}\subset
E$ such that $\mbox{cap}(U_{\varepsilon})<\varepsilon$ and
$f|_{E\setminus U_{\varepsilon}}\in L^1(E\setminus
U_{\varepsilon};m)$. In \cite{KR:JFA} it is proved that if $f$ is
quasi-integrable in the analytic sense then it is
quasi-integrable.
\end{uw}

We say that a real process $M$ is a local martingale additive
functional (local MAF) of $\mathbb{X}$ if it is an additive
functional of $\mathbb{X}$ (see \cite[Section 5.1]{Fukushima}) and
$M$ is a local martingale under $P_x$ (with respect to the
filtration $\FF$) for each $x\in E\setminus N$, where $N$ is an
exceptional set of $M$.

We would like to emphasize  that the notion of local MAF differs
from the notion of MAF  locally of finite energy considered in
\cite[Section 5.5]{Fukushima}. For instance, $M$ having the last
property is local AF, i.e. is additive on $[0,\zeta)$ only.

Let us consider the following system
\begin{equation}
\label{eq3.1}
-Au=F(x,u)+\mu.
\end{equation}

\begin{df}
We say that a function $u:E\rightarrow\BR^N$ is a solution of
(\ref{eq3.1}) if
\begin{enumerate}
\item[(a)] $u$ is quasi-continuous and $u\in\mathbf{D}$,
\item[(b)] $u(X_{t\wedge\zeta})\rightarrow 0$ as $t\rightarrow\infty$ $P_x$-a.s.
for q.e. $x\in E$,
\item[(c)] $E\ni x\mapsto F(x,u(x))\in qL^1(E;m)$,
\item[(d)] there exists a local ($N$-dimensional) MAF $M$ of $\BX$
such that for q.e. $x\in E$ and every $T>0$,
\begin{align}
\label{eq3.2}
u(X_t)&=u(X_{T\wedge\zeta})+\int_t^{T\wedge\zeta}F(X_r,u(X_r))\,dr+
\int_t^{T\wedge\zeta}dA_r^\mu\nonumber\\
&\quad-\int_t^{T\wedge\zeta}dM_r,\quad t\in[0, T\wedge\zeta],\quad
P_x\mbox{-a.s.}
\end{align}
\end{enumerate}
\end{df}

\begin{uw}
\label{uw.fkf} Observe that if $u:E\rightarrow \BR^N$ is a
measurable function  such that
$E_{x}\int_{0}^{\zeta}|F(X_{r},u(X_r))|\,dr<\infty$ and
\[
u(x)=E_{x}\int_{0}^{\zeta}F(X_{r},u(X_r))\,dr
+E_{x}\int_{0}^{\zeta}dA^{\mu}_{r}
\]
for q.e. $x\in E$, then $u$ is a solution of (\ref{eq3.1}).
Indeed,  by the Markov property,
\[
u(X_t)=E_x\Big(\int_{t}^{\zeta}F(X_{r},u(X_r))\,dr+\int_{t}^{\zeta}dA^{\mu}_{r}
|\FF_t\Big),\quad t\in[0,\zeta].
\]
From this it is easily seen that $u\in \mathbf{D}$ and $u$
satisfies (b). It is also clear that (c) is satisfied. That $u$ is
quasi-continuous it follows from \cite[Lemma 4.2]{KR:JFA}. Now let
us put
\[
M^x_{t}=E_x\Big(\int_0^{\zeta} F(r,u(X_r))\,dr +\int_0^{\zeta}
dA^{\mu}_{r}|\FF_{t}\Big)-u(X_0),\quad t\ge 0.
\]
By \cite[Lemma A.3.5]{Fukushima} there exists a c\`adl\`ag process
$M$ such that
\[P_x(M_t=M^x_t,\,\, t\ge 0)=1\]
for q.e. $x\in E$. It is clear that $M$ is a MAF of $\mathbb{X}$
and (d) is satisfied.
\end{uw}

We first show that if $F$ is monotone, i.e. $F$ satisfies the
condition
\begin{enumerate}
\item[(H5)] $\langle F(x,y)-F(x,y'),y-y'\rangle\le 0$, $y,y'\in\BR^N$,
$x\in E$,
\end{enumerate}
then the probabilistic solution of ({\ref{eq3.1}) is unique.

In the sequel for a given $x\in \BR^N$ such that $x\neq 0$ we
write
\[\hat{\mbox{sgn}}(x)=\frac{x}{|x|}\,.\]

\begin{stw}
\label{prop3.2} Assume that $\mbox{\rm{(H5)}}$ holds. Then there
exists at most one solution of $\mbox{\rm{(\ref{eq3.1})}}$.
\end{stw}

\begin{dow}
Let $u_{1}, u_{2}$ be solutions of (\ref{eq3.1}) and $M_{1},
M_{2}$ be local MAFs associated with $u_{1}, u_{2}$, respectively.
Put $u=u_{1}-u_{2}$ and $M=M_{1}-M_{2}$. Then
\[
u(X_{t})=u(X_{\tau\wedge\zeta})+\int_{t}^{\tau\wedge\zeta}
(F(\cdot,u_{1})-F(\cdot,u_{2}))(X_{r})\,dr
-\int_{t}^{\tau\wedge\zeta}\,dM_{r},\quad 0\le
t\le\tau\wedge\zeta,\, P_x\mbox{-a.s.}
\]
for every bounded $\tau\in\mathcal{T}$ and q.e. $x\in E$. By the
Tanaka-Meyer formula and (H5), for q.e. $x\in E$ we have
\begin{align*}
|u(X_{t})|&\le |u(X_{\tau\wedge\zeta})|+\int_{t}^{\tau\wedge\zeta}
\langle F(\cdot,u_{1})-F(\cdot,u_{2})(X_{r}), \hat{\mbox{sgn}}
(u(X_{r}))\rangle\,dr\\&\quad
-\int_{t}^{\tau\wedge\zeta}\langle\hat{\mbox{sgn}}(u(X)_{r-}),dM_{r}\rangle\\
&\le |u(X_{\tau\wedge\zeta})|-\int_{t}^{\tau\wedge\zeta}
\langle\hat{\mbox{sgn}}(u(X)_{r-}),dM_{r}\rangle,\quad 0\le
t\le\tau\wedge\zeta, \quad P_x\mbox{-a.s.}
\end{align*}
Let $\{\tau_{k}\}$ be a fundamental sequence for the local
martingale
$\int_{0}^{\cdot\wedge\zeta}\langle\hat{\mbox{sgn}}(u(X)_{r-})
,dM_{r}\rangle$. Putting $t=0$ in the above inequality with $\tau$
replaced by $\tau_k$ and then taking the expectation with respect
to $P_x$ we get
\begin{equation}
\label{eq3.111} |u(x)|\le E_{x}|u(X_{\tau_k\wedge\zeta})|
\end{equation}
for q.e. $x\in E$.  Since $u\in\mathbf{D}$,  letting
$k\rightarrow\infty$ we conclude that $|u|=0$ q.e.
\end{dow}

\begin{tw}
\label{tw.1} Assume that $(\BX,\BB_1,m)$ has the compactness
property and $\mbox{\rm{(H1)--(H4)}}$ are satisfied. Then there
exists a solution of \mbox{\rm(\ref{eq3.1})}.
\end{tw}
\begin{dow}
Step 1. We first assume that $\|R|\mu|\|_\infty<\infty$ and there
exists  a strictly positive bounded Borel measurable function $g$
such that $|F(x,y)|\le g(x)$ for $x\in E$, $y\in\BR^N$ and
$\|Rg\|_\infty<\infty$. Let $\rho$ be a strictly positive Borel
measurable function on $E$ such that $\int \rho(x)m(dx)<\infty$
and let
\[
\Phi:L_2(E;\rho\cdot m)\rightarrow L_2(E;\rho\cdot m),\quad
\Phi(u)=RF(\cdot,u)+R\mu.
\]
The mapping $\Phi$ is well defined since $|R(F(\cdot,u))+R\mu|\le
Rg+R|\mu|\in L^2(E;\rho\cdot m)$. By (H3), $\Phi$ is  continuous.
We shall show that $\Phi$ is compact. To see this, let us consider
 $\{u_n\}\subset L^2(E;\rho\cdot m)$. By Remark \ref{uw.fkf}, the  function
$v_n=\Phi(u_n)$ is a probabilistic solution of the system
\[
-Av_n=F(x,u_n)+\mu.
\]
Therefore there is a MAF $M^n$ of $\BX$ such that
\begin{align}
\label{eq.a1}
v_n(X_t)&=v_n(X_{T\wedge\zeta})+\int_t^{T\wedge\zeta}F(X_r,u_n(X_r))\,dr+
\int_t^{T\wedge\zeta}dA_r^\mu\nonumber\\
&\quad-\int_t^{T\wedge\zeta}dM_r^n, \quad t\in[0,\zeta\wedge
T],\quad P_x\mbox{-a.s.}
\end{align}
for q.e. $x\in E$. Hence
\[|v_n(x)-p_tv_n(x)|=|v_n(x)-E_xv_n(X_{t\wedge\zeta})|\le
E_x\int_0^{t\wedge\zeta}g(X_r)\,dr+
E_x\int_0^{t\wedge\zeta}dA_r^{|\mu|}
\]
for q.e. $x\in E$. Consequently,
\[
\lim_{t\rightarrow 0^+}\sup_n|p_tv_n(x)-v_n(x)|=0
\]
for q.e. $x\in E$. Observe that $\|v_n\|_{\infty}\le
\|Rg\|_{\infty}+\|R|\mu|\|_{\infty}$. Since
$(\BX,\mathcal{B}_1,m)$ has the compactness property, it follows
from Corollary \ref{wn2.7} that there is a subsequence
$(n')\subset(n)$ such that $\{v_{n'}\}$ converges q.e. Since
$\{v_{n'}\}$ are uniformly bounded by
$\|Rg\|_\infty+\|R\mu\|_\infty$, applying the Lebesgue dominated
convergence theorem shows that $\{v_{n'}\}$ converges in
$L^2(E;\rho\cdot m)$. By Schauder's fixed point theorem, there is
$u\in L^2(E;\rho\cdot m)$ such that $\Phi(u)=u$, i.e.
\[
u(x)=E_x\int_0^\zeta F(X_r,u(X_r))\,dr+E_x\int_0^\zeta dA_r^\mu
\]
for $m$-a.e. $x\in E$. Let $v(x)$  be equal to the right-hand side
of the above equality for $x\in E$  such that
$Rg(x)+R|\mu|(x)<\infty$ and zero otherwise. Then by
\cite[Lemma 4.2]{KR:JFA}, $v$ is quasi-continuous and $v\in\mathbf{D}$. Since
$v=u$, $m$-a.e., we have
\[
E_x\int_0^\zeta F(X_r,u(X_r))\,dr =E_x\int_0^\zeta
F(X_r,v(X_r))\,dr
\]
for q.e. $x\in E$. Thus $v$ is a solution of  (\ref{eq3.1}) (see Remark \ref{uw.fkf}).
\smallskip\\
Step 2. Now we consider the general case. Let $g$ be a strictly
positive bounded Borel measurable function on  $E$ such that
$\|Rg\|_\infty<\infty$ (for the existence of $g$ see
\cite[Corollary 1.3.6]{Oshima}) and let $\{F_n\}$ be a generalized
nest such that $\|R|\mu_n|\|_\infty$, where
$\mu_n=\mathbf{1}_{F_n}\cdot\mu$. Put
\[
F_n(x,y)=\frac{ng(x)}{1+ng(x)}\cdot\frac{n\cdot
F(x,y)}{|F(x,y)|\vee n}, \quad x\in E,\,y\in\BR^N.
\]
Then  $F_n$ satisfies (H2)--(H4) and $R|F_n|\le n^2Rg$, which
implies that $\|R|F_n|\|_\infty<\infty$. By Step 1, for each $n\ge
1$ there exists a solution $u_n$ of the system
\[-Au_n=F_n(x,u_n)+\mu_n.
\]
Therefore there is a MAF $M$ of $\BX$ such that
\begin{align}
\label{eq3.3} \nonumber u_n(X_t)=&u_n(X_{T\wedge\zeta})
+\int_t^{T\wedge\zeta}F_n(X_r,u_n(X_r))\,dr
+\int_t^{T\wedge\zeta}dA_r^{\mu_n}\\
&-\int_t^{T\wedge\zeta}dM_r^n, \quad t\in[0,T\wedge\zeta],\quad
P_x\mbox{-a.s.}
\end{align}
for q.e. $x\in E$. By the Tanaka-Meyer formula,
\begin{align*}
|u_n(X_t)|&\le|u_n(X_{T\wedge\zeta})|
-\int_t^{T\wedge\zeta}\langle\hat{\mbox{sgn}}(u_n(X)_{r-}),
F_n(X_r,u_n(X_r))\rangle\,dr\\
&+\int_t^{T\wedge\zeta}
\langle\hat{\mbox{sgn}}(u_n(X)_{r-}),dA_r^{\mu_n}\rangle
-\int_t^{T\wedge\zeta} \langle
\hat{\mbox{sgn}}(u_n(X)_{r-}),dM_r^n\rangle, \quad
t\in[0,T\wedge\zeta].
\end{align*}
By the above inequality and (H4),
\[
|u_n(x)|\le E_x|u_n(X_{T\wedge\zeta})|+E_x\int_0^{T\wedge\zeta}G(X_r)\,dr
+E_x\int_0^{T\wedge\zeta}dA_r^{|\mu|}
\]
for q.e. $x\in E$. Letting $T\rightarrow\infty$ and using the fact
that $u_n\in\mathbf{D}$ we conclude that for q.e. $x\in E$,
\begin{equation}
\label{eq3.4} |u_n(x)|\le E_x\int_0^\zeta G(X_r)\,dr
+E_x\int_0^\zeta\,dA_r^{|\mu|}.
\end{equation}
Put $v(x)=E_x\int_0^\zeta
G(X_r)\,dr+E_x\int_0^\zeta\,dA_r^{|\mu|}$ if the right-hand side
of (\ref{eq3.4}) is finite and $v(x)=0$ otherwise. By
\cite{KR:JFA}, $v$ is quasi-continuous, $v\in\mathbf{D}$ and $v$
is a probabilistic solution of the equation
\begin{equation}
\label{eq3.5}
-Av=G+|\mu|.
\end{equation}
Let $U_k=\{v<k\}$. Since $v$ is quasi-continuous, $U_k$ is finely
open. Moreover, since by (H1) and (H4) $v$ is finite,
$\bigcup^\infty_{k=1}U_k=E$ q.e. Write $\tau_k=\tau_{U_k}$. Then
\begin{equation}
\label{eq3.6}
|u_n\mathbf{1}_{U_k}(x)|\le k,\quad n\ge 1,\quad x\in U_k.
\end{equation}
By (H2),
\begin{equation}
\label{eq3.7} P_x\Big(\int_0^{T\wedge\zeta}\sup_{|y|\le k}
|F|(X_r,y)\,dr<\infty,T>0\Big)=1
\end{equation}
for every $k\ge 0$. Let
\[
\sigma_k=\inf\{t>0;\quad\int_0^t\sup_{|y|\le k}|F|(X_r,y)\,dr>k\}.
\]
By (\ref{eq3.7}), $\sigma_k\nearrow\infty$. Let
$\delta_{k,l}=\tau_k\wedge\sigma_l$. By (\ref{eq3.3}),
(\ref{eq3.6}), (\ref{eq3.7}) and the construction of
$\delta_{k,l}$ we have
\begin{align*}
|u_n(x)-E_xu_n(X_{t\wedge\delta_{k,l}\wedge\zeta})|
&\le E_x\int_0^{t\wedge\delta_{k,l}\wedge\zeta}
|F_n|(X_r,u_n(X_r))\,dr
+E_x\int_0^{t\wedge\delta_{k,l}\wedge\zeta}\,dA_r^{|\mu|}\\
&\le kE_x(t\wedge\delta_{k,l}\wedge\zeta)
+E_x\int_0^{t\wedge\delta_{k,l}\wedge\zeta}\,dA_r^{|\mu|}.
\end{align*}
Hence
\begin{equation}
\label{eq3.8} \lim_{t\rightarrow
0^+}\sup_n|u_n(x)-E_xu_n(X_{t\wedge\delta_{k,l}\wedge\zeta})|=0
\end{equation}
for q.e. $x\in E$. Now we will show that (\ref{eq3.8}) holds for $x\in U_k$ with
$E_x|u_n(X_{t\wedge\delta_{k,l}\wedge\zeta})|$ replaced by
$E_x[|u_n(X_{t\wedge\tau_k})|\mathbf{1}_{\{t<\tau_k\}}]$. To this
end, let us first observe that $P_x(\tau_k>0)=1$ for $x\in U_k$,
because $U_k$ is finely open. We have
\begin{align*}
&\sup_n|E_xu_n(X_{t\wedge\delta_{k,l}})
-E_xu_n(X_t)\mathbf{1}_{\{t<\tau_k\}}|\le
E_x\sup_n|u_n(X_{t\wedge\delta_{k,l}})
-u_n(X_t)\mathbf{1}_{\{t<\tau_k\}}|\\
&\qquad=\int_{\{t\ge\tau_k\}\cup\{t\ge\delta_{k,l}\}}
\sup_n|u_n(X_{t\wedge\delta_{k,l}})
-u_n(X_t)\mathbf{1}_{\{t<\tau_k\}}|\,dP_x\\
&\qquad\le\int_{\{t\ge\tau_k\}\cup\{t\ge\delta_{k,l}\}}
|v(X_{t\wedge\delta_{k,l}})|+|v(X_t)\mathbf{1}_{\{t<\tau_k\}}|\,dP_x.
\end{align*}
Since $\lim_{t\rightarrow0^+}\lim_{l\rightarrow\infty}
P_x(\{t\ge\tau_k\}\cup\{t\ge\delta_{k,l}\})=0$ for $x\in U_k$ and
$v\in\mathbf{D}$, it follows that for $x\in U_k$ the right-hand
side of the above inequality tends to zero as
$l\rightarrow+\infty$ and then $t\rightarrow 0^+$. This and
(\ref{eq3.8}) imply that
\begin{equation}
\label{eq3.9}
\lim_{t\rightarrow 0^+}\sup_n|u_n(x)-p_t^ku_n(x)|=0,\quad x\in U_k,
\end{equation}
where $\{p_t^k,t\ge 0\}$ is the semigroup associated with the
process $\BX^{U_k}$. By Proposition \ref{stw.2} the triple
$(\BX^{U_k},\BB_1(U_k),m)$ has the compactness property. Moreover,
$\BX^{U_k}$ is normal since $U_k$ is finely open. Therefore it
follows from  Theorem \ref{stw2.1} and (\ref{eq3.6}) that there
exists a subsequence $(n')\subset(n)$ such that
$\{u_{n'}\mathbf{1}_{U_k}\}$ is convergent q.e. By using standard
argument and the fact that $\bigcup_kU_k=E$ q.e. one can now
construct a subsequence $(m)\subset (n)$ such that $\{u_m\}$ is
convergent q.e. on $E$. Without loss of generality we may assume
that $(m)=(n)$. Let us write $u=\limsup u_n$ and
$\delta_k=\delta_{k,k}$. By (\ref{eq3.3}),
\begin{align*}
u_n(X_{t\wedge\delta_k})&=E_x(u_n(X_{T\wedge\delta_k})
+\int_{t\wedge\delta_k\wedge\zeta}^{T\wedge\delta_k\wedge\zeta}
F_n(X_r,u_n(X_r))\,dr\\
&\quad+
\int_{t\wedge\delta_k\wedge\zeta}^{T\wedge\delta_k\wedge\zeta}\,dA_r^{\mu_n}
|\FF_{t\wedge\delta_k\wedge\zeta}), \quad t\in[0,T],\quad
P_x\mbox{-a.s.},
\end{align*}
so applying  \cite[Lemma 6.1]{BDHPS}  we can conclude that for
every $q\in(0,1)$,
\begin{align*}
&E_x\sup_{t\le\delta_k\wedge T}|u_n(X_t)-u_m(X_t)|^q
\le\frac{1}{1-q}E_x\Big[|u_n(X_{\delta_k\wedge T})
-u_m(X_{\delta_k\wedge T})|^q\\
&\quad+\Big(\int_0^{T\wedge\delta_k\wedge\zeta}
|F_n(X_r,u_n(X_r))-F_m(X_r,u_m(X_r))|\,dr\Big)^q
+\Big(\int_0^{T\wedge\delta_k\wedge\zeta}dA_r^{|\mu_n-\mu_m|}\Big)^q
\Big].
\end{align*}
Applying the Lebesgue dominated convergence theorem and using
(H3), the construction of $F_n,\{\delta_k\}$ and the convergence
of $\{u_n\}$ we conclude that for q.e. $x\in E$ the first and
second term on the right-hand side of the above inequality
converges to zero as $n,m\rightarrow\infty$. To show the
convergence of the third term, let us  observe that
\[
A_t^{|\mu_n-\mu_m|}=\int_0^t\mathbf{1}_{F_n\Delta
F_m}(X_r)\,dA_r^\mu,\quad t\ge0.
\]
Since $E_x\int_0^{\zeta}dA_r^{|\mu|}<\infty$ q.e., it is enough to
show that
\[
\lim_{n,m\rightarrow\infty}P_x(\exists_{t>0}X_t\in F_n\Delta
F_m)=\lim_{n,m\rightarrow\infty}P_x(\sigma_{F_n\Delta F_m}<\infty)=0
\]
for q.e. $x\in E$. But this follows immediately from the fact that
$\{F_n\}$ is a nest (see \cite[Theorem 3.4.8]{Oshima}). By what has already
been proved,
\[
\big(u_n(X),\int_0^\cdot F_n(X_r,u_n(X_r))\,dr,
A^{\mu_n}\big)\rightarrow \big(u(X),\int_0^\cdot
F(X_r,u(X_r))\,dr,A^{\mu}\big),
\]
uniformly on compacts in probability $P_x$ for q.e. $x\in E$.
Therefore letting $n\rightarrow\infty$ in (\ref{eq3.3}) we see
that there exists a local MAF $M$ of $\BX$ such that (\ref{eq3.2})
is satisfied for q.e. $x\in E$. The fact that $u\in \mathbf{D}$
and $u$ satisfies condition (b) of the definition of a
probabilistic solution of (\ref{eq3.1}) follows from (\ref{eq3.4})
and (\ref{eq3.5}).
\end{dow}

\nsubsection{Systems with operators generated by right Markov processes}
\label{sec4}

In the present section we assume that $\BX$ is a general transient
right Markov process on $E$ satisfying hypothesis (L) of Meyer.

Let us fix an excessive ($\sigma$-finite) measure $m$ on $E$, i.e.
a Borel measure on $\BB(E)$ such that
\[
m\circ\alpha R_{\alpha}\le m,
\]
where $(m\circ\alpha R_{\alpha})f=m(\alpha R_{\alpha}f)=\int
f(x)\,m(dx)$ for $f\in\BB^+(E)$.

We say that a set $B\subset E$ is $m$-polar if there exists an
excessive function $v$ such that $A\subset\{v=\infty\}$ and $v$ is
finite $m$-a.e.

In this section we say that a property holds q.e. if it holds
except for some $m$-polar set.

Recall that a set $N\in\BB^n(E)$ is $m$-inessential if it is
$m$-polar and absorbing for $\BX$.

\begin{df}
An $\FF$-adapted increasing $[0,\infty]$-valued process
$\{A_{t},t\ge 0\}$ is called positive co-natural additive
functional (PcNAF) of $\BX$ if there exist a defining set
$\Omega_A\subset\FF_\infty$ and an $m$-inessential Borel set
$N_A\subset E$ such that
\begin{enumerate}
\item[(a)] $P_x(\Omega_A)=1$ for $x\notin N_A$ and
$\theta_t\Omega_A\subset\Omega_A$, $t\ge 0$,
%\item[(b)] $\theta_t\Omega_A\subset\Omega_A$, $t\ge 0$,
\item[(b)] for every $\omega\in\Omega_A$ the mapping
$t\mapsto A_t(\omega)$ is right continuous on $[0,\infty)$ and
finite valued on $[0,\zeta)$ with $A_0(\omega)=0$,
\item[(c)] for every $\omega\in\Omega_A$ and $t>0$,
$\Delta A_t\equiv A_t-A_{t-}=a(X_t)$, where $a\in p\BB^n(E)$,
\item[(d)] for every
$w\in\Omega_A$, $A_{t+s}(\omega)=A_t(\omega)+A_s(\theta_t\omega)$
for all $s,t\ge 0$.
\end{enumerate}
\end{df}

\begin{uw}
\label{uw.0} It is known (see \cite[Proposition 6.12]{GS}) that
for any $m$-polar set $N$ there exists a Borel $m$-inessential set
$B$ such that $N\subset B$. Therefore if some property holds q.e.
then without loss of generality we may assume that it holds
everywhere except for possibly an $m$-inessential set.
\end{uw}

Given a PcNAF $A$ and $f\in\BB^+(E)$ set
\[
U_Af(x)=E_x\int_0^\zeta f(X_r)\,dA_r,\quad x\in E.
\]
By $\mu_A$ we denote the Revuz measure associated with $A$, i.e.
the measure defined as
\[
\mu_A(f)=\sup\{\nu\circ U_Af;\quad\nu\circ U\le m\}.
\]

In this section by a nest we understand  an increasing sequence
$\{B_n\}$ of nearly Borel sets such that
$P_m(\lim_{n\rightarrow\infty}\tau_{B_n}<\zeta)=0$.

\begin{df}
A Borel measure $\mu$ on $E$ is called smooth if it charges no
$m$-polar sets and there exists a nest $\{G_n\}$ of finely open
nearly Borel sets such that $\mu(G_n)<\infty$, $n\ge 1$.
\end{df}

It is known (see \cite[Therems 6.15, 6.21, 6.29]{FG}) that for
every PcNAF $A$ its Revuz measure $\mu_A$ is smooth and for every
smooth measure $\mu$ there exists a unique PcNAF $A^\mu$ such that
its Revuz measure is equal to $\mu$.

\begin{lm}
\label{lm4.1} Let $\mu$ be a positive smooth measure. Then the
function $u$ defined as
\[
u(x)=E_x\int_0^\zeta\,dA^\mu,\quad x\in E
\]
is finely continuous and if $u<\infty$, $m$-a.e. then $u<\infty$
q.e.
\end{lm}

\begin{dow}
That $u$ is finely continuous follows from \cite[Theorems 36.10, 49.9]{Sharpe}. Since
$u$ is finely continuous, $F=\{u=\infty\}$ is finely closed.
Therefore
\begin{align*}
P_m(\sigma_F<\infty)=P_m(\sigma_F<\infty,X_{\sigma_F}\in F)
&=P_m(\sigma_F<\infty,u(X_{\sigma_F})=\infty)\\
&= P_m\Big(\sigma_F<\infty, E_x\big(\int_{\sigma_E}^\infty
dA_r^\mu|\FF_{\sigma_F}\big)\Big)\\
&\le P_m\Big(\sigma_F<\infty, E_x\big(\int_0^\zeta
dA_r^\mu|\FF_{\sigma_F\wedge\zeta}\big)=\infty\Big)\\
&\le P_m\Big(E_x\big(\int_0^\zeta
dA_r^\mu|\FF_{\sigma_F\wedge\zeta}\big)=\infty\Big)=0,
\end{align*}
which when combined with \cite[Corollary 1.8.6]{BB} implies that
$F$ is $m$-polar.
\end{dow}

\begin{stw}
\label{stw4.44} $(\BX,\BB_1,m)$ has the compactness property iff
$(\BX,[0,v],m)$ has the compactness property for every
$v\in\mathbf{D}$.
\end{stw}
\begin{dow}
Sufficiency is obvious. To prove necessity, let us assume that
$(\BX,\BB_1,m)$ has the compactness property and for $v\in
\mathbf{D}$ let us choose $\{u_n\}\subset\mathcal{B}^+(E)$ such
that $u_n\le v$, $m$-a.e. for $n\ge 1$. Write $u^k_n=T_k(u_n)$.
Since $v\in \mathbf{D}$, $R_\alpha v$ is finite $m$-a.e.  for
every $\alpha>0$. Let $g$ be a strictly positive Borel measurable
function on $E$ such that $\int(R_\alpha v)g\,dm<\infty$. By the
assumption, for every $k\ge 0$ there exists a subsequence
$(n')\subset (n)$ such that $\{R_\alpha u^k_{n'}\}$ is convergent
in $L^1(E; g\cdot m)$. Therefore to show the existence of a
subsequence $(m)\subset (n)$ such that $R_\alpha u_m$  converges
in $L^1(E; g\cdot m)$ it is enough to prove that $\|R_\alpha
u^k_n-R_\alpha u_n\|_{L^1(E; g\cdot m)}\le C(k)$ for some
independent of $n$ constants $C(k)$ such that $C(k)\rightarrow 0$
as $k\rightarrow +\infty$. Observe that
\begin{align*}
\|R_\alpha u^k_n-R_\alpha u_n\|_{L^1(E; g\cdot m)}
&\le E_{g\cdot m}\int_0^{\infty} e^{-\alpha t}|u_n(X_t)-u^k_n(X_t)|\,dt\\
&\le E_{g\cdot m}\int_0^{\infty}
e^{-\alpha t}\mathbf{1}_{\{v(X_t)>k\}}v(X_t)\,dt\equiv C(k).
\end{align*}
Since  $\int(R_\alpha v)g\,dm<\infty$, $C(k)\rightarrow 0$ as
$k\rightarrow \infty$.
\end{dow}

\begin{stw}
Let $\BX$ be a right Markov process and $m$ be an excessive
measure. Then $(\BX,\BB_1,m)$ has the compactness property iff
$R_\alpha:L^1(E;m)\rightarrow L^1(E;m)$ is order compact for some
(and hence for every) $\alpha>0$.
\end{stw}

\begin{dow}
Necessity. Assume that $(\BX,\BB_1(E),m)$ has the compactness
property. Let $v\in L^1(E;m)$ and $\{u_n\}\subset[0,v]$. Write
$v_n^k=R_\alpha u_n^k$, $u_n^k=T_k(u_n)$ for $n,k\ge 1$.  By the
compactness property of $(\BX,\BB_1(E),m)$ there exists a
subsequence (still denoted by $(n)$) such that $v_n^k$ is $m$-a.e.
convergent. Since $v_n^k\le R_\alpha v\in L^1(E;m)$, $v_n^k$ is
convergent in $L^1(E;m)$. Furthermore,
\begin{align*}
\|v_n^k-v_n\|_{L^1(E;m)}\le\int_ER_\alpha|T_k(u_n)-u_n|\,dm
&\le\frac 1 \alpha \int_E|T_k(u_n)-u_n|\,dm\\
&\le\frac 2 \alpha \int_{\{v>k\}} v\,dm\equiv C(k).
\end{align*}
Since $C(k)\rightarrow0$ as $k\rightarrow+\infty$, there exists a
subsequence $(n')\subset(n)$ such that $\{v_{n'}\}$ is convergent
in $L^1(E;m)$.

Sufficiency. Now assume that $R_\alpha:L^1(E;m)\rightarrow
L^1(E;m)$ is order compact. Let $\{u_n\}\subset\mathcal{B}^+(E)$
be such that $u_n(x)\le 1$ for $x\in E,n\ge 1$. Let $\{g_k\}$ be a
sequence of positive functions in $L^1(E;m)$ such that
$g_k\nearrow 1$ and let $\rho$ be a strictly positive function in
$L^1(E;m)$. Write $u^k_n=u_n g_k$, $v_n^k=R_\alpha u^k_n$,
$v_n=R_\alpha u_n$. By the assumption, for every $k\ge1$ there
exists a subsequence (still denoted by $n$) such that $v^k_n$
converges in $L^1(E;m)$. It follows that for every $k\ge1$ there
exists a subsequence (still denoted by $n$) such that $v_n^k$
converges in $L^1(E; \rho\cdot m)$. This when combined with the
fact that
\[
\lim_{k\rightarrow\infty}\|v^k_n-v_n\|_{L^1(E;\rho\cdot m)}\le
\lim_{k\rightarrow\infty} \int_E R_\alpha |1-g_k|\rho\,dm=0
\]
implies the existence of a subsequence $(n')\subset (n)$ such that
$v_{n'}$ converges in $L^1(E;\rho\cdot m)$. Therefore there is a
further subsequence $(n'')\subset (n')$ such that $v_{n''}$
converges $m$-a.e.
\end{dow}
\medskip

Let us consider the following system
\begin{equation}
\label{eq4.1} -Au=F(x,u)+\mu,
\end{equation}
where $(A,(D(A))$ is the operator defined by
\begin{equation}
\label{eq4.2} D(A)=R_\alpha(L^1(E;m)),\quad-A(R_\alpha f)=f-\alpha
R_\alpha f,\quad f\in L^1(E;m).
\end{equation}
for some $\alpha>0$. Since $m$ is an excessive measure,
\[
\int_E\alpha R_\alpha f\,dm\le\int_Ef\,dm,\quad f\in\BB^+(E),
\]
from which it follows immediately that $R_\alpha f=0$, $m$-a.e.,  if $f=0$, $m$-a.e. Therefore (\ref{eq4.2}) makes sense. Also
note that by the resolvent equation the definition of $(A,D(A))$
is independent of $\alpha>0$.

\begin{stw}
\label{uw.1} Let $B\in\BB(E)$. If $(\BX,\BB_1(E))$ has the
compactness property then $(\BX^B,\BB_1(E))$ has the
compactness property.
\end{stw}
\begin{proof}
Follows by the same method as in the proof of Proposition
\ref{stw.2}, because under the assumption of the present
proposition we need not use  Lemma \ref{lm1.1}.
\end{proof}

\begin{tw}
\label{tw4.1} Let $\BX$ be a transient Markov process satisfying
condition  $\mbox{\rm{(L)}}$ of Meyer. Assume that
$\mbox{\rm{(H1)--(H4)}}$ are satisfied. Then there exists a
solution of $\mbox{\rm{(\ref{eq4.1})}}$.
\end{tw}

\begin{dow}
We assume that there exists a Borel function $g\in\BB^+(E)$ such
that $|F(x,y)|\le g(x)$ for $x\in E$, $y\in\BR^N$ and $Rg$ is
finite $m$-a.e. Let $\rho$ be a strictly positive Borel function
on $E$ such that
\[\int_E[E_xA_\zeta^{|\mu|}+Rg(x)]^2\rho(x)\,m(dx)<\infty.
\]
Let
\[
\Phi:L_2(E;\rho\cdot m)\rightarrow L_2(E;\rho\cdot m),\qquad
\Phi(u)=RF(\cdot,u)+R\mu.
\]
The mapping $\Phi$ is well defined since
\[|R(F(\cdot,u))+R\mu|\le Rg\in R|\mu|\in L^2(E;\rho\cdot m).
\]
In fact $\Phi:L^2(E;\rho\cdot m)\rightarrow B_{L^2(E;\rho\cdot
m)}(0,r)$, where $r=\|Rg\|_{L^2(E;\rho\cdot
m)}+\|R|\mu|\|_{L^2(E;\rho\cdot m)}$. $\Phi$ is continuous by
(H3). Let $\{u_n\}\subset L^2(E;\rho\cdot m)$. Define $v_n$ by
putting $v_n(x)=RF(\cdot,u_n)(x)+R\mu(x)$ for $x$ such that
$Rg(x)+R|\mu|(x)<\infty$ and $v_n(x)=0$ otherwise. By the
assumptions and Lemma \ref{lm4.1}, $v_n$ is finely continuous and
finite q.e. By Remark \ref{uw.0} we may assume that it is finite
except for an $m$-inessential set. Then by the strong Markov
property formula (\ref{eq.a1}) holds. Therefore repeating the
arguments following (\ref{eq.a1}) and applying Proposition
\ref{stw4.44} we conclude that $\Phi$ is compact. The rest of the
proof now runs as in Step 2 of the proof of Theorem \ref{tw.1} (we
use Proposition \ref{uw.1} instead of Proposition \ref{stw.2}).
\end{dow}

\nsubsection{Applications}
\label{sec5}

In this section we give several examples of processes having the
compactness property.
\begin{prz}
Let $\{\mu_t,t>0\}$ be a convolution semigroup on $\BR^d$ and let
$\BX$ be a Hunt process with the transition function
\[
p_tf(x)=\int_{\BR^d}f(x+y)\mu_t(dy).
\]
It is known (see \cite{Bertoin}) that if for some $\varepsilon>0$
\begin{equation}
\label{eq2.12} \lim_{|x|\rightarrow
\infty}|x|^{-\varepsilon}|\mbox{Re}\,\psi|(x)\rightarrow \infty,
\end{equation}
where $\hat{\mu}_t(x)=e^{-t\psi(x)},\,x\in\BR^d$ ($\hat{\mu}_t$
stands for  the  Fourier transform of $\mu_t$) then the Lebesgue
measure $m$ on $\BR^d$ is a reference measure for $\BX$. Therefore
if $\BX$ is a L\'evy process with the characteristic exponent
$\psi$ satisfying (\ref{eq2.12}) then $(\BX,\BB_1)$ has the
compactness property. Consequently, our existence and uniqueness
results of Section \ref{sec3} (Theorem \ref{tw.1} and Proposition
\ref{prop3.2}) apply to systems with operator $A$ of the form
$\psi(\nabla)$ with $\psi$ satisfying (\ref{eq2.12}). A model
example is $\psi$ of the form $\psi(x)=|x|^{\alpha},x\in\BR^d$,
for some $\alpha\in(0,2]$, which corresponds to the fractional
Laplacian $\psi(\nabla)=(\nabla^2)^{\alpha/2}=\Delta^{\alpha/2}$.
\end{prz}

\begin{prz}
Let $H$ be a real Hilbert space, $Q\in\mathcal{L}(H)$ be a
selfadjoint nonnegative operator and $A$ be  a generator of a
$C_0$-semigroup $e^{tA}$ on $H$. Let
\[
Q_t=\int_0^te^{sA}Qe^{sA^*}\,ds
\]
be of trace class, $e^{tA}(H)\subset Q_t^{1/2}(H)$ and
$\mbox{Ker\,}Q_t=\{0\}$, $t>0$. It is well known that the
Ornstein-Uhlenbeck semigroup
\[
(T_t\phi)(x)=\int_H\phi(x)\mathcal{N}(e^{tA}x,Q_t)\,dy, \quad
\phi\in\BB_b(E),
\]
where $\mathcal{N}(e^{tA}x,Q_t)$ is the Gaussian measure on $H$
with mean $e^{tA}x$ and covariance operator $Q_t$ is representable
by the Ornstein-Uhlenbeck process being a solution of the  SDE
\[
\left\{
\begin{array}{l}dX(t,x)=AX(t,x)\,dt+Q^{1/2}dW(t)\smallskip \\
X(0,x)=x\in H,
\end{array}
\right.
\]
i.e.
\[
(T_t\phi)(x)=E_x\phi(X_t),\quad \phi\in\BB_b(E)
\]
(see \cite{DaPratoZabczyk} for details). By the Cameron-Martin
formula (see, e.g., \cite{DaPrato}), $\mathbb{X}$ satisfies
Meyer's hypothesis (L), which implies that $(\mathbb{X},\BB_1)$
has the compactness property. Therefore results of Section
\ref{sec3} apply to systems with Ornstein-Uhlenbeck operator being
a generator of the semigroup $\{T_t\}$.
\end{prz}

\begin{prz}
Let $(\mathcal{E}, D[\mathcal{E}])$ be a regular symmetric
Dirichlet form on $L^2(E;m)$. By \cite{Fukushima}, if the
following Sobolev type inequality holds
\[
\|u\|_{p_0}^2\le c\EE_{\lambda_0}(u,u),\quad u\in D[\EE]
\]
for some $c>0$, $p_0>2$, $\lambda_0\ge 0$, then $m$ is a reference
measure for $\BX$ associated with $(\EE,D[\EE])$. Consequently,
$(\mathbb{X},\BB_1)$ has the compactness property.
\end{prz}

\begin{prz}
Let $(\EE, D[\EE])$ be a regular semi-Dirichlet form and let $\BX$
be the  associated Hunt process. Suppose that $(\BX,\PP,m)$ has
the compactness  property. Let $\mu$ be a positive smooth measure
and let $(\EE_\mu,D[\EE_\mu])$ be  the form defined as
\[
\EE_\mu(u,v)=\EE(u,v)+\int_Euv\,d\mu,\quad
D[\EE_\mu]=\{u\in D[\EE];\,\int_E|u|^2\,d\mu<\infty\}.
\]
It is known that $(\EE_\mu,D[\EE_\mu])$ is a quasi-regular
Dirichlet form and that the associated standard special
process $\BX^\mu$ is a subprocess of $\BX$ (see \cite[Section
6.4]{Fukushima}). Therefore $(\BX^{\mu},\PP,m)$ has the
compactness property.
\end{prz}

\begin{prz}
Let $(\{X_t, t\ge 0\},\{P_{s,x},(s,x)\in[0,\infty)\times E\})$ be
a  time inhomogenous Markov process. Assume that for every $s\ge
0$ the pair $(X^s=(\{X_{s+t},t\ge 0\},\{P_{s,x},x\in
E\}),\BB_1(E))$ has the compactness property. Then by Proposition
\ref{stw2.2}, $X^s$ has a reference measure $m(s)$. Assume that
$m(s)=m$, $s\ge 0$. Let $\mathbf{Z}=(\{Z_t,t\ge
0\},\{P_z,z\in[0,\infty)\times E\})$, where
$Z_t=(\tau(t),X_{\tau(t)})$, $t\ge 0$ and $\tau$ is the uniform
motion to the right i.e. $\tau(t)=\tau(0)+t$,
$P_{s,x}(\tau(0)=s)=1$. Then $\mathbf{Z}$ is a Markov process with
reference measure $\bar{m}=dt\otimes m$. Indeed, we have
\begin{equation}
\label{eq5.2} R^{\mathbf{Z}}_\alpha((s,x),T\times B)
=\int_0^\infty e^{-\alpha t}E_{s,x}\mathbf{1}_B(X_{s+t})
\cdot\mathbf{1}_T(s+t)\,dt.
\end{equation}
Suppose that $T\in\mathcal{B}([0,\infty)), B\subset\BB(E)$ and
$\bar{m}(T\times B)=0$. Then  $dt(T)=0$ or  $m(B)=0$. If $dt(T)=0$
then it is clear that $R_\alpha^{\mathbf{Z}}((s,x),T\times B)=0$.
If $m(B)=0$ then the right-hand side of (\ref{eq5.2}) is less then
or equal to
\[
\int_0^\infty e^{-\alpha t}E_{s,x}\mathbf{1}_B(X_{s+t})\,dt
=R_{\alpha}^{\BX^s}(x, B)=0,
\]
the last equality being a consequence of the fact that
$R^{\BX^s}_\alpha(x,dy)\ll m(dy)$. Thus $\bar m$ is the reference
measure. As a result, the pair
$(\mathbf{Z},\mathcal{B}_1([0,\infty)\times E))$ has the
compactness property. For instance, let $\{A(t),t\ge 0\}$ be a
family of operators associated with regular semi-Dirichlet forms
$\EE^{(t)}$ on $L^2(E;m)$ and let $\mathbf{Z}$ be a process
associated with the operator $\LL=\frac{\partial}{\partial
t}+A(t)$. If for every $t\ge 0$ the  Hunt processes associated
with $\EE^{(t)}$ together with  $\mathcal{B}_1(E)$ form pairs
having the compactness property with the same reference measure
then the pair $(\mathbf{Z}, \BB_1([0,\infty)\times E))$ has the
compactness property.
\end{prz}

\begin{prz}
Let $\BX$ be a solution of the following $d$-dimensional SDE
\[
X_t^x=x+\sum^d_{j=1}\int_0^ta_j(r,X_r^x)\,dW_r^j+\int_0^tb(r,X_r^x)\,dr,
\]
where $x\in\BR^d$ and
$a_j,b:[0,\infty]\times\BR^d\rightarrow\BR^d$, $j=1,\dots,d$, are
measurable functions satisfying the assumptions
\begin{enumerate}
\item[(a)] $\sum_{j=1}^d|a_j(t,x)-a_j(t,y)|+|b(t,x)-b(t,y)|
\le L|x-y|$ for every $x,y\in\BR^d$, $t\ge 0$,
\item[(b)] $t\rightarrow a_j(t,0)$,
$t\rightarrow b(t,0)$ are bounded on $[0,T]$ for every $T>0$.
\end{enumerate}
Then by \cite[Theorem 2.3.2]{Nualart}, if
\[
P(S_x=0)=1,
\]
where
\[
S_x=\inf\{t>0;\,\int_0^t\mathbf{1}_{\{\det\sigma(r,X_r)\neq
0\}}\,dr>0\}\wedge T,
\]
then for every $t>0$ the distribution of  $X_t$ is absolutely
continuous with respect to the Lebesgue measure on $\BR^d$. It
follows that if, for instance, $\sigma(t,x)>0$ for every
$(t,x)\in[0,\infty)\times\BR^d$, then $(\BX,\mathcal{B}_1(\BR^d))$
has the compactness property. More generally, let $A$ be an
absorbing set for $\BX$, i.e. if $x\notin A$ then
$P(\exists_tX_t^x\in A)=0$. Then if $\sigma(t,x)>0$ for every
$(t,x)\in[0,\infty)\times \BR^d\setminus A$ then
$(\BX^{\BR^d\setminus A},\mathcal{B}_1(E\setminus A))$ has the
compactness property. To be more specific, let us consider
diffusion process describing dividend-paying asset prices in the
classical multidimentional Black and Scholes model, i.e.
\[
X_t^{x,i}=x_i+\int_0^t(r-d_i)X_r^{x,i}\,dr
+\sum^d_{j=1}\int_0^t\sigma_{ij}X_r^{x,i}\,dW_r^j,\quad
i=1,\dots,d.
\]
Then $(\BX^{\BR^d\setminus A},\mathcal{B}_1(\BR^d\setminus A))$,
where $ A=\{x\in\BR^d: x_i=0\mbox{ for some }i=1,\dots,d\}$, has
the compactness property.
\end{prz}

We close this section with an example of a right Markov process
$\BX$ which is not associated with a Dirichlet form,  so that the
results of Section \ref{sec3} can not be applied to systems with
operator associated with $\BX$. However, $\BX$ satisfies Meyer's
hypothesis (L), so that results of Section \ref{sec4} are
applicable.

\begin{prz}
For $\phi\in C^1(\BR^2)$, $x\in\BR^2$ set
\[
L\phi(x)=\frac12 \mbox{tr}(QD^2\phi(x))+\langle Ax, D\phi (x)\rangle,
\]
where
\[
Q=\begin{bmatrix}
1&1\\1&1
\end{bmatrix},\quad A=\begin{bmatrix}
-1&0\\0&-2
\end{bmatrix}.
\]
Then the semigroup $e^{tA}$  generated by $A$ is of the form
\[e^{tA}=\begin{bmatrix}
e^t&0\\0&e^{-2t}
\end{bmatrix},\quad t\ge 0
\]
and
\[
Q_t\equiv \int_0^t e^{sA}Qe^{sA}\,ds=\begin{bmatrix}
\frac12(1-e^{-2t})&\frac13 (1-e^{-3t})\smallskip \\
\frac13(1-e^{-3t})&\frac14(1-e^{-4t}),
\end{bmatrix},\quad Q_{\infty}=   \begin{bmatrix}
\frac12&\frac13\smallskip\\  \frac13&\frac14
\end{bmatrix}.\]
It is clear that $\mbox{Ker\,}Q_t=\{0\}$ and $Q_t>0$ for every
$t>0$. Let $\{P_t,\, t\ge 0\}$ be the semigroup generated by the
operator $L$ on $L^2(\BR^2;\mu)$, where $\mu=\mathcal{N}(0,Q_\infty)$.
 It is well known that
\[
P_tf(x)=E_xf(X_t),
\]
where $X$ is a unique solution of the  SDE
\[
dX_t=AX_t\,dt+Q^{1/2}\,dW_t,\quad X_0=x.
\]
From \cite{Kuptsov} it follows that $\mathbb{X}=\{(X,P_x),
x\in\BR^2\})$ satisfies Meyer's hypthesis (L). Therefore
$(\mathbb{X},\mathcal{B}_1)$ has the compactness property. On the
other hand, by \cite{Goldys}, $\{P_t,t\ge 0\}$ is variational
(i.e. is associated with a Dirichlet form on $L^2(\BR^2,\mu)$) if
and only if $\{P_t,t\ge 0\}$ is analytic. By \cite{Goldys} (see
also \cite{Fuhrman}), $\{P_t,t\ge 0\}$ is analytic if and only if
$Q$ is invertible. Accordingly, $\{P_t, t\ge 0\}$ is not
variational.
\end{prz}
\vspace{3mm} \noindent{\bf\large Acknowledgements}
\medskip\\
Research supported by Polish NCN grant no. 2012/07/D/ST1/02107.

\end{document}